\documentclass[12pt,a4paper,leqno]{elsarticle}

\usepackage[utf8]{inputenc}
\usepackage[T1]{fontenc}
\usepackage[english]{babel}
\usepackage{amsthm}
\usepackage{amscd}
\usepackage{amsfonts}         
\usepackage{amsmath}
\usepackage{amssymb}
\usepackage{hyperref}
\usepackage{enumerate}
\usepackage{tikz}
\usepackage{tikz-cd}
\usepackage{mathrsfs}  

\newcommand{\Ab}{\mathbb{A}}
\newcommand{\Bb}{\mathbb{B}}

\newcommand{\Gb}{\mathbb{G}}

\newcommand{\Lb}{\mathbb{L}}

\newcommand{\Pb}{\mathbb{P}}
\newcommand{\Qb}{\mathbb{Q}}

\newcommand{\Zb}{\mathbb{Z}}

\newcommand{\Cc}{\mathcal{C}}

\newcommand{\Ec}{\mathcal{E}}
\newcommand{\Fc}{\mathcal{F}}

\newcommand{\Ic}{\mathcal{I}}

\newcommand{\Mc}{\mathcal{M}}
\newcommand{\Nc}{\mathcal{N}}
\newcommand{\Oc}{\mathcal{O}}
\newcommand{\Pc}{\mathcal{P}}

\newcommand{\Rc}{\mathcal{R}}

\newcommand{\Vc}{\mathcal{V}}

\newcommand{\Cs}{\mathscr{C}}

\newcommand{\Fs}{\mathscr{F}}

\newcommand{\Is}{\mathscr{I}}

\newcommand{\Ls}{\mathscr{L}}
\newcommand{\Ms}{\mathscr{M}}

\newcommand{\Ss}{\mathscr{S}}

\newcommand{\PCob}{{\underline\Omega}}

\newcommand{\coim}{\mathrm{\coim}}

\newcommand{\Spec}{\mathrm{Spec}}

\newcommand{\Fl}{\mathrm{Fl}}
\newcommand{\bl}{\mathrm{Bl}}
\newcommand{\Td}{\mathrm{Td}}
\newcommand{\wtil}{\widetilde}

\newcommand{\colim}{\mathrm{colim}}

\newcommand{\LSym}{\mathrm{LSym}}

\newcommand{\abs}[1]{\lvert #1 \rvert}

\newcommand{\rank}{\mathrm{rank}}

\newcommand{\hook}{\hookrightarrow}
\newcommand{\cl}{\mathrm{cl}}
\newcommand{\modmod}{/\!\!/}
\newcommand{\inv}{\mathrm{inv}}

\newtheorem{theo}{Tplottin ubuntuheorem}[section]
\theoremstyle{plain}
\newtheorem{thm}[theo]{Theorem}
\newtheorem{lem}[theo]{Lemma}
\newtheorem{prop}[theo]{Proposition}
\newtheorem{cor}[theo]{Corollary}
\newtheorem*{thm*}{Theorem}
\newtheorem*{lem*}{Lemma}
\newtheorem*{prop*}{Proposition}
\newtheorem*{cor*}{Corollary}

\theoremstyle{definition}
\newtheorem{defn}[theo]{Definition}

\newtheorem{ex}[theo]{Example}
\newtheorem{cons}[theo]{Construction}
\newtheorem{rem}[theo]{Remark}

\newtheorem{quest}[theo]{Question}
\newtheorem*{defn*}{Definition}

\date{}
\begin{document}

\begin{frontmatter}

\title{Base independent algebraic cobordism}
\author{Toni Annala}
\affiliation{organization={Department of Mathematics, University of British Columbia},%Department and Organization
            addressline={1984 Mathematics Rd}, 
            city={Vancouver},
            postcode={V6T1Z2}, 
            state={British Columbia},
            country={Canada}}

\begin{abstract}
The purpose of this article is to show that the bivariant algebraic $A$-cobordism groups considered previously by the author are independent of the chosen base ring $A$. This result is proven by analyzing the bivariant ideal generated by the so called snc relations, and, while the alternative characterization we obtain for this ideal is interesting by itself because of its simplicity, perhaps more importantly it allows us to easily extend the definition of bivariant algebraic cobordism to divisorial Noetherian derived schemes of finite Krull dimension. As an interesting corollary, we define the corresponding homology theory called algebraic bordism. We also generalize projective bundle formula, the theory of Chern classes, the Conner--Floyd theorem and the Grothendieck--Riemann--Roch theorem to this setting. The general definitions of bivariant cobordism is based on the careful study of ample line bundles and quasi-projective morphisms of Noetherian derived schemes, also undertaken in this work.
\end{abstract}


\end{frontmatter}

\tableofcontents

\section{Introduction}

Algebraic cobordism is supposed to be the universal oriented cohomology theory in algebraic geometry. While still largely hypothetical, it promises to be a valuable invariant in algebraic geometry: Voevodsky's original approach to proving Milnor conjecture was based on the algebraic cobordism spectrum $MGL$, Levine and Pandharipande used in \cite{levine-pandharipande} the classification of varieties up to cobordism over a field of characteristic 0 in their proof of degree 0 Donaldson--Thomas conjectures, and Sechin and Sechin--Semenov have found applications of algebraic Morava K-theories (defined using algebraic cobordism) to approximating torsion in Chow groups \cite{sechin-chern} and to arithmetic questions about algebraic groups \cite{sechin-semenov}. Moreover, finding the correct definition of algebraic cobordism is intimately connected with fundamental questions in algebraic geometry, such as the existence of a good Chow cohomology theory (and more generally a good bivariant intersection theory).

After the foundational treatise \cite{Levine:2007} by Levine and Morel, algebraic cobordism has attracted a lot of attention. The groups have gone through several redefinitions and generalizations in \cite{levine-pandharipande}, \cite{lowrey-schurg}, \cite{annala-cob}, \cite{annala-yokura}, \cite{annala-pre-and-cob}, which, starting from \cite{lowrey-schurg} make serious use of derived algebraic geometry. In the most recent work \cite{annala-pre-and-cob} it was shown that given a Noetherian ring $A$ of finite Krull dimension, one can construct bivariant $A$-cobordism groups $\Omega^*_A(X \to Y)$ for all morphisms $X \to Y$ of quasi-projective derived $A$-schemes, and that these groups generalize all the previous geometric models of algebraic cobordism on quasi-projective derived schemes. This still leaves a lot to hope for: first of all, it is not clear how these groups depend on the base ring $A$, neither is it clear how the corresponding homology theory --- algebraic bordism --- fits into this picture as the algebraic $A$-bordism groups $\Omega^A_*(X) := \Omega^{-*}_A\big(X \to \Spec(A) \big)$ are not expected to model algebraic bordism unless $A$ is regular. Moreover, it is conceivable that one would like to study the cobordism of schemes that are not quasi-projective over a base. The work presented in this article deals with these issues, among others.

% and a morphism of quasi-projective derived $A$-schemes, one can construct bivariant $A$-cobordism groups $\Omega^*_A(X \to Y),$ which contain essentially all previous definitions of algebraic cobordism as special case: for instance, if $A=k$ is a field of characteristic $0$, there are natural isomorphisms $\Omega^*_k(X \to \Spec(k)) = \Omega^\mathrm{LM}_*(\tau_0(X))$ for all quasi-projective derived $k$-schemes $X$, where $\Omega^\mathrm{LM}_*$ denotes the algebraic bordism of Levine--Morel constructed in \cite{Levine:2007}.

\subsection{Summary of results}

Let us then summarize the contributions of this paper. Our first main result is the following base independence result.
\begin{thm*}[Corollary \ref{BivACobIndFromACor}]
Let $X \to Y$ be a morphism of derived schemes with $X$ and $Y$ quasi-projective over Noetherian rings $A$ and $B$ of finite Krull dimension. Then
$$\Omega^*_A(X \to Y) = \Omega^*_B(X \to Y).$$
\end{thm*}
\noindent In order to prove this result, we have to study the bivariant ideal of snc relations $\langle \Rc^\mathrm{snc}_A \rangle \subset \PCob^*_A$ as $\Omega^*_A := \PCob^*_A/\langle \Rc^\mathrm{snc}_A \rangle$, where $\PCob^*_A$ is the universal precobordism theory over $A$. In Definition \ref{SNCIdealAltDef} and Theorem \ref{BivACobAltConsThm} we find explicit descriptions of the groups $\langle \Rc^\mathrm{snc}_A \rangle(X \to Y)$ and Corollary \ref{BivACobIndFromACor} is an immediate consequence of this description. Besides helping us to prove the desired base independence, and being interesting for its own sake, the simplified construction of bivariant algebraic cobordism also generalizes easily to give rise to bivariant algebraic cobordism $\Omega^\bullet$, which is bivariant theory defined on the $\infty$-category of finite dimensional Noetherian derived schemes admitting an ample family of line bundles (also called divisorial derived schemes, see Definition \ref{AmpleFamDef}). Note that the groups $\PCob^\bullet(X \to Y)$ do not admit a grading in general, which is why we have adopted the notation $\PCob^\bullet$ instead of $\PCob^*$.  We also state pleasant universal properties for the bivariant theories $\PCob^\bullet$ and $\Omega^\bullet$ (Theorem \ref{UnivPropOfPCobThm} and Corollary \ref{UnivPropOfCobCor}) analogous to the universal properties of $\PCob^*_A$ and $\Omega^*_A$ proved in \cite{annala-pre-and-cob}.

The second main result of this article is the definition of algebraic bordism in general, which is the oriented Borel--Moore homology version of algebraic cobordism, and the generalization of Levine--Morel's algebraic bordism outside the realm of characteristic 0 geometry. Since the bivariant theory $\Omega^\bullet$ is defined so generally, we can simply use the following definition.
\begin{defn*}[Definition \ref{AlgBordDef}]
Let $X$ be a divisorial and finite dimensional Noetherian derived scheme. We define the \emph{algebraic bordism} of $X$ as the Abelian group
$$\Omega_\bullet(X) := \Omega^\bullet \big(X \to \Spec(\Zb)\big).$$
In other words $\Omega_\bullet(X)$ classifies complete intersection derived schemes (see Section \ref{DerComplIntSubSect}) mapping projectively to $X$ ``up to cobordism''. 
\end{defn*}
\noindent Of course, as hinted by the above definition, we could have defined $\Omega_\bullet(X)$ directly by itself, but realizing it as a part of a stably oriented bivariant theory automatically verifies that it has good formal properties and that it interacts well with the algebraic cobordism ring $\Omega^\bullet(X)$. Note that the above definition uses in crucial way the fact that the bivariant groups $\Omega^\bullet(X \to Y)$ are defined even when $X \to Y$ is not a morphism of finite type. If $A$ is a regular Noetherian ring and $X$ is quasi-projective over $A$, then $\Omega_\bullet(X)$ is isomorphic to $\Omega_*^A(X)$, but this not true in general. We note that it might be more correct to consider the modified bordism group $\Omega'_\bullet(X)$ (see the next paragraph) whenever $X$ fails to have an ample line bundle. 

The third main result of this paper is the generalization of the results of \cite{annala-chern} to divisorial derived schemes. The first obstacle is that the old proof of nilpotence of Euler classes of vector bundles does not seem to generalize to derived schemes without an ample line bundle, leading us in Section \ref{NilpEulerClassSubSect} to define the modified bivariant theories $\PCob'^\bullet$ and $\Omega'^\bullet$, for which the desired nilpotence can be proven. Most of Section \ref{GeneralPBFSect} is then dedicated to proving the projective bundle theorem.
\begin{thm*}[Theorem \ref{GeneralPBFThm}]
Let $\Bb^\bullet$ be a bivariant quotient theory of $\PCob'^\bullet$, let $X \to Y$ be a morphism of finite dimensional divisorial Noetherian derived schemes, and let $E$ be a vector bundle of rank $r$ on $X$. Then the morphism
$$\Pc roj: \bigoplus_{i=0}^{r-1} \Bb^\bullet(X \to Y) \to \Bb^\bullet(\Pb(E) \to Y)$$
given by
$$(\alpha_0, \alpha_1, ..., \alpha_r) \mapsto \sum_{i=0}^{r-1} e(\Oc(1))^i \bullet 1_{\Pb(E)/X} \bullet \alpha_i$$
is an isomorphism.
\end{thm*}
\noindent Even though the statement of the theorem is a straightforward generalization of the corresponding result in \cite{annala-chern}, the result itself does not come for free: the basic idea of the proof is still the same, but the details have to be considerably altered for derived schemes without an ample line bundle. Moreover, our modified proof has the advantage that we do not have to define and study Chern classes before proving the projective bundle formula; rather, we construct a good theory of Chern classes in $\PCob'^\bullet$ (see Definition \ref{GeneralChernClassDefn} and Theorem \ref{GeneralChernClassThm}) using the argument of Grothendieck based on the projective bundle formula. The following result is a standard consequence of a good theory of Chern classes.
\begin{thm*}[Theorem \ref{GeneralCFThm}]
Suppose $X$ is a finite dimensional Noetherian derived scheme admitting an ample family of line bundles. Then the morphism
$$\Zb_m \otimes_\Lb \PCob'^\bullet(X) \to K^0(X)$$
given by
$$[V \stackrel f \to X] \mapsto [Rf_* \Oc_V]$$
is an isomorphisms of rings. The theorem remains true if we replace $\PCob'^\bullet$ by $\Omega'^\bullet$.
\end{thm*}
\noindent We also generalize the Grothendieck--Riemann--Roch theorem in Theorem \ref{GeneralGRRThm}. Since having an ample family of line bundles is a much weaker condition that being quasi-projective over an affine base (for instance, all regular schemes are divisorial), the results of Section \ref{GeneralPBFSect} of this paper considerably generalize those of \cite{annala-chern}. Note that this is not just abstract nonsense either: if $X$ is a divisorial (e.g. regular) non-quasi-projective variety over a field $k$ of positive characteristic, then none of the previous work on algebraic cobordism proves such a Conner--Floyd theorem above for $X$. We also note that divisorial derived schemes is almost as far as one could conceivably generalize the current proof of Conner--Floyd theorem: it might be possible to generalize everything to a derived scheme $X$ having the resolution property, but without this assumption, $K^0(X)$ is no longer generated by vector bundles on $X$, completely breaking down the construction of the Chern character.

Finally, in order to define our bivariant theories in this generality, we had to carefully define and study ample line bundles an quasi-projective morphisms of derived schemes (see Sections \ref{AmpleSubSect} and \ref{QProjSubSect}). The results obtained are of course very similar to the classical story, but generalizing them to derived geometry proved to be very useful for us, and hopefully will prove to be useful for others as well. Such results for derived schemes, as far as the author is aware, have not appeared anywhere else (except weaker versions in the unpublished note \cite{annala-qpnote}).

\subsection{Conventions}

We will freely use the language of $\infty$-categories and derived schemes whenever necessary. Given an additive sheaf $\Fc$ on a derived scheme $X$, we will denote its \emph{global sections} by $\Gamma(X; \Fc)$: this is naturally regarded as a spectrum object in simplicial Abelian groups or equivalently as a chain complex of Abelian groups. We will denote by $\abs{\Gamma(X; \Fc)}$ the underlying (connective) simplicial set of $\Gamma(X; \Fc)$. The cohomology groups are defined as
$$H^i(X; \Fc) := \pi_{-i} \big( \Gamma(X; \Fc)\big);$$
if $\Fc$ is a discrete sheaf, then these recover the usual cohomology groups of sheaves, while in general they are more closely related to hypercohomology. If $\Fc$ is discrete, we denote by $\Gamma_\cl(X; \Fc)$ the classical global sections of $\Fc$. 

A derived scheme $X$ has an underlying classical scheme $X_\cl$ which is also called the \emph{truncation}. A morphism $f: X \to Y$ of derived schemes is \emph{of finite type} (\emph{proper}) if $f_\cl$ is of finite type (proper) in the classical sense. Given a vector bundle $E$ on a derived scheme $X$, the projective bundle $\Pb(E) \to X$ classifies all the subbundles of $E$ of rank $1$, while, if $\Fc$ is a connective quasi-coherent sheaf on $X$, $\mathbf{P}(\Fc) := \Pc roj(\LSym^*_X(\Fc))$ classifies all surjections (maps inducing a surjection of sheaves on $\pi_0$) from $\Fc$ to line bundles. If $A$ is a Noetherian ring of finite Krull dimension, then a quasi-projective derived $A$-scheme $X$ is implicitly assumed to be Noetherian. All derived schemes are assumed to be separated, and to have finite Krull-dimensional Noetherian underlying space. Given a morphism $X \to Y$ of derived schemes, we will denote by $\Lb_{X/Y}$ its \emph{relative cotangent complex}. In the case where $Y \simeq \Spec(\Zb)$ this is denoted by $\Lb_{X}$ and called the \emph{absolute cotangent complex}.

Given an object $X$ in an $\infty$-category $\Cc$ with a distinguished final object $pt$, we will use the shorthand notation $\pi_X$ to denote the essentially unique morphism $X \to pt$. Bivariant orientations are denoted by $1_{X/Y}$ instead of $\theta(X \to Y)$ used in previous work. We will denote by $\Lb$ the Lazard ring. Given a formal group law $F(x,y)$, we will denote by $\inv_F(x)$ the formal inverse power series.

\subsection*{Acknowledgements}
The author is supported by Vilho, Yrj\"o and Kalle V\"ais\"al\"a Foundation of the Finnish Academy of Science and Letters.

\section{Background}
\subsection{Derived algebraic geometry}

The construction of our cohomology theories makes use of derived algebraic geometry. The main references for the basic results in derived algebraic geometry are the work of Toën--Vezzosi \cite{HAG1,HAG2} and the work of Lurie \cite{HTT,HA,SAG}, but the reader can also consult the background section of \cite{annala-yokura} where a concise summary of useful results is presented. Derived schemes are Zariski locally modeled by homotopy types of simplicial commutative algebras. Since we only consider derived schemes whose underlying topological space is Noetherian and of finite Krull dimension, we can use the following vanishing result.

\begin{lem}[Grothendieck vanishing]\label{GrothendieckVanishingLem}
Let $X$ be a Noetherian topological space with Krull dimension $n < \infty$, and let $\Fc$ be a sheaf of simplicial Abelian groups on $X$. Then, for all $j \in \Zb$, 
$$\pi_j\big(\Gamma(X; \Fc)\big) = \pi_j\big(\Gamma(X; \tau_{\leq i}(\Fc))\big)$$
as soon as $i \geq n+j$.
\end{lem}
\begin{proof}
Using the fundamental cofibre sequences
$$\Gamma\big(X; \pi_{i+1}(\Fc)[i+1] \big) \to \Gamma\big(X; \tau_{\leq i+1}(\Fc) \big) \to \Gamma\big(X; \tau_{\leq i}(\Fc) \big)$$
and the classical Grothendieck vanishing, we conclude that $\pi_j\big( \Gamma(X; \tau_{\leq i}(\Fc)) \big)$
stabilize for $i \gg 0$, and therefore (by Milnor sequence)
$$\pi_j \big( \Gamma(X; \Fc) \big) = \lim_i \pi_j\big(\Gamma(X; \tau_{\leq i}(\Fc))\big).$$
The bound for $i$ follows from the fact that a Noetherian topological space of Krull dimension $n$ has cohomological dimension $n$.
\end{proof}

Moreover, for technical reasons it is often useful to restrict our attention to Noetherian schemes:

\begin{defn}\label{NoetherianDef}
A derived scheme $X$ is \emph{Noetherian} if $X_\cl$ is Noetherian as a classical scheme and all the homotopy sheaves $\pi_i(\Oc_X)$ are coherent $\Oc_{X_\cl}$-modules. 
\end{defn}

\begin{defn}\label{CoherentDef}
A quasi-coherent sheaf $\Fc$ on a Noetherian derived scheme $X$ is called \emph{almost perfect} if the homotopy sheaves $\pi_i(\Fc)$ are coherent sheaves on $X_\cl$, and they vanish for $i \ll 0$. An almost perfect sheaf $\Fc$ is \emph{coherent} if all but finitely many homotopy sheaves $\pi_i(\Fc)$ vanish.
\end{defn}

\subsubsection{Ample line bundles and ample families}\label{AmpleSubSect}

The purpose of this section to define and to study ample line bundles and schemes admitting an ample family of line bundles in the context of derived algebraic geometry. The results in this section generalize the well known classical results as well as the (weaker) results of the unpublished note \cite{annala-qpnote}. Let us start with the fundamental definitions. 

\begin{defn}\label{AmpleDef}
Let $X$ be a Noetherian derived scheme and let $\Ls$ be a line bundle on $X$. We say that $\Ls$ is \emph{ample} if for any point $x \in X$ there exists $n \geq 0$ and a section $s \in \Gamma(X; \Ls^{\otimes n})$ such that $X_s := X \backslash V(s)$ is affine and contains $x$.
\end{defn}

\begin{defn}\label{AmpleFamDef}
Let $X$ be a Noetherian derived scheme. We say that $X$ is \emph{divisorial} (or $X$ \emph{has an ample family of line bundles}) if for any point $x \in X$ there exists a line bundle $\Ls$ and a global section $s \in \Gamma(X; \Ls)$ such that $X_s$ is affine and contains $x$.
\end{defn}

The following Lemma will be important when analyzing the ramifications of the above definitions.

\begin{lem}\label{LocLem}
Let $X$ be a Noetherian derived scheme, let $\Ls$ be a line bundle on $X$ and let $s$ be a global section of $\Ls$. Let $\Fc$ be a quasi-coherent sheaf on $X$, and consider the sequences
$$H^i(X; \Fc) \xrightarrow{s \cdot} H^i(X; \Ls \otimes \Fc) \xrightarrow{s \cdot} H^i(X; \Ls^{\otimes 2} \otimes \Fc) \xrightarrow{s \cdot} \cdots$$
Then there are natural isomorphisms
$$\colim_n H^i \big(X; \Ls^{\otimes n} \otimes \Fc \big) \xrightarrow{\cong} H^i(X_s; \Fc)$$
given by sending $f \in H^i(X; \Ls^{\otimes n} \otimes \Fc)$ to $f/s^n \in H^i(X_s; \Fc)$.
\end{lem}
\begin{proof}
It is enough to prove this for $i=0$. Since applying the truncation $\tau_{\geq 0}$ does not change the zeroth cohomology, we can assume that $\Fc$ is connective. Finally, we can use Grothendieck vanishing (Lemma \ref{GrothendieckVanishingLem}) to reduce to the situation where the nontrivial homotopy sheaves of $\Fc$ lie between degrees $0$ and $n$, where $n$ is at most the Krull dimension of $X$. Our argument proceeds by induction on $n$. Notice that the base case $n=0$ is classical (see e.g. \cite{Stacks} Tag 09MR). To prove the general case, we consider the cofibre sequence
$$\Fc' \to \Fc \to \pi_0(\Fc).$$
Since the nontrivial homotopy sheaves of $\Fc'$ lie between degrees $1$ and $n$, we can apply the induction assumption to $\Fc'$ and $\pi_0(\Fc)$, and the claim follows from 5-lemma applied to the diagram comparing the cohomology long exact sequences (note that sequential colimits preserve exact sequences).
\end{proof}

We can now prove the following useful characterization of ampleness.

\begin{prop}\label{AmpleCharProp1}
Let $X$ be a Noetherian derived scheme, and let $\Ls$ be a line bundle on $X$. Then the following are equivalent
\begin{enumerate}
\item $\Ls$ is ample;
\item for every connective almost perfect sheaf $\Fc$ on $X$, the sheaves $\Ls^{\otimes n} \otimes \Fc$ are globally generated for $n \gg 0$.
\end{enumerate}
\end{prop}
\begin{rem}
We recall that a connective quasi-coherent sheaf $\Fc$ on $X$ is \emph{globally generated} if there exists a set $I$ and a morphism
$$\Oc_X^{\oplus I} \to \Fc$$
inducing a surjection on $\pi_0$.
\end{rem}
\begin{proof}
Using Grothendieck vanishing (Lemma \ref{GrothendieckVanishingLem}) we can assume that $\Fc$ is coherent. Let us first assume that $\Ls$ is ample, and let us begin by choosing $s_i \in \Gamma(X; \Ls^{\otimes n_i})$, $i=1..r$, so that $X_{s_i}$ form an affine open cover of $X$. As 
$$\pi_0\big(\Gamma(X_{s_i}; \Fc)\big) \cong \Gamma\big(X_{s_i}; \pi_0(\Fc) \big),$$
we can use Lemma \ref{LocLem} to conclude that for $d_i \gg 0$ there exist global sections of $\Ls^{\otimes d_i n_i} \otimes \Fc$ generating it at the points of $X_{s_i}.$ Letting $N = n_1 \cdots n_r$, it then follows that $\Ls^{\otimes iN} \otimes \Fc$ is globally generated for $i \gg 0$. Since the last conclusion is also true for the sheaves $\Ls \otimes \Fc, ..., \Ls^{\otimes N-1} \otimes \Fc$, it follows that $\Ls^{\otimes n} \otimes \Fc$ is globally generated for $n \gg 0$.

Suppose then that for all coherent sheaves $\Fc$, $\Ls^{\otimes n} \otimes \Fc$ is globally generated for $n \gg 0$. Let $x \in X$ be a point, and let $U$ be an affine open neighbourhood of $x$ s.t. $\Ls \vert_U$ is trivial. Let $Z \hook X$ be the complement of $U$ with the reduced scheme structure, and consider the cofibre sequence
$$\Ic \to \Oc_X \to \Oc_Z.$$
Supposing that $\Ls^{\otimes n} \otimes \Ic$ is globally generated, it is possible to find a global section $s$ so that its image in $s' \in \Gamma(X; \Ls^{\otimes n})$ does not vanish at $x$. It follows that $X_{s'} \subset U$ is an affine open subset containing $x$.
\end{proof}

Admitting an ample family of line bundles has an alternative characterization as well.

\begin{prop}\label{AmpleFamCharProp}
Let $X$ be a Noetherian derived scheme. Then the following are equivalent
\begin{enumerate}
\item $X$ is divisorial;
\item for any connective almost perfect sheaf $\Fc$ on $X$, there exists a surjection
$$\Ls_1 \oplus \Ls_2 \oplus \cdots \oplus \Ls_r \to \Fc$$
from a direct sum of line bundles. 
\end{enumerate}
\end{prop}
\begin{proof}
Suppose first that $X$ admits an ample family of line bundles. Choose global sections $s_1, ... , s_r$ of line bundles $\Ls_1, ..., \Ls_r$ such that $X_{s_i}$ form an affine open covering for $X$. Given a connective almost perfect sheaf $\Fc$ on $X$, we can use Lemma \ref{LocLem} to find $n_i \geq 0$ and global sections $s_{i1},...,s_{id_i}$ which generate $\Ls^{\otimes n_i}_i \otimes \Fc$ at the points of $X_{s_i}$. Said otherwise, we find a surjective morphism
$$\big( \Ls^{\otimes - n_1}_1 \big)^{\oplus d_1} \oplus \cdots \oplus \big( \Ls^{\otimes - n_r}_r \big)^{\oplus d_r} \to \Fc,$$
which is exactly what we wanted.

To prove the converse, let $U$ be an affine open neighborhood of $x$ with reduced complement $Z$, and consider the cofibre sequence
$$\Ic \to \Oc_X \to \Oc_Z$$
of almost perfect sheaves on $X$. Given a surjective morphism
$$\Ls_1 \oplus \cdots \oplus \Ls_r \to \Ic,$$
then for some $1 \leq i \leq r$ the induced map $\Ls_i \to \Oc_X$ is surjective at $x$. In other words $\Ls_i^\vee$ has a global section $s$ such that $x \in X_s$. As $X_s \subset U$ and $U$ is affine, also $X_s$ is affine.
\end{proof}

Next we consider the relative version of ampleness.

\begin{defn}\label{RelAmpleDef}
Let $f: X \to Y$ be a morphism of derived schemes. We say that a line bundle $\Ls$ on $X$ is \emph{relatively ample} (or \emph{$f$-ample}) if for every affine open set $\Spec(B) \to Y$, $\Ls$ restricts to an ample line bundle on the inverse image $f^{-1} \Spec(B)$.
\end{defn}

if $Y$ is affine, then a line bundle $\Ls$ on $X$ is relatively ample over $Y$ if and only if $\Ls$ is ample. Combining this observation with the following lemma allows us to check relative ampleness on an affine cover of $Y$.

\begin{lem}\label{RelAmpleLocalLem}
Let $f: X \to Y$ be a morphism of Noetherian derived schemes, let $(U_i)_{i \in I}$ be an open cover of $Y$, and let $\Ls$ be a line bundle on $X$.  Then $\Ls$ is $f$-ample if and only if for all $i \in I$, we have that $\Ls_i := \Ls \vert_{f^{-1} U_i}$ is $f_i$-ample, where $f_i: f^{-1} U_i \to U_i$ is the restriction of $f$.
\end{lem}
\begin{proof}
The only if direction is clear, so let us prove that $\Ls_i$ being $f_i$-ample implies the relative ampleness of $\Ls$. In other words we have to prove the following: if $Y = \Spec(B)$ is affine, $b_1, ..., b_r \in B$ are functions such that $(Y_{b_i})_{i=1}^r$ is an affine open cover of $Y$ with $\Ls \vert_{X_{b_i}}$ ample for all $i$, then $\Ls$ is ample. But proving this is easy: if $x \in X$, then $x$ lies inside one of the open sets $X_{b_i}$, and we can use Lemma \ref{LocLem} to find a global section $s$ of $\Ls^{\otimes n}$ with $X_{s} \subset X_{b_i}$ affine and containing $x$.
\end{proof}

As a consequence, we prove that relatively ample line bundles behave well in compositions and in derived pullbacks.

\begin{prop}\label{AmpleLineBundlesInCompositionsLem}
Let $f: X \to Y$ be a morphism of Noetherian derived schemes having a relatively ample line bundle $\Ls$. Then
\begin{enumerate}
\item if $g: Y \to Z$ and $\Ls$ is $(g \circ f)$-ample, then $\Ls$ is $f$-ample;
\item if $g: Y \to Z$ admits a relatively ample line bundle $\Ms$, then $\Ls \otimes f^* \Ms^{\otimes n}$ is $(g \circ f)$-ample for $n \gg 0$;
\item if 
$$
\begin{tikzcd}
X' \arrow[]{r}{f'} \arrow[]{d}{h'} & Y' \arrow[]{d}{h} \\
X \arrow[]{r}{f} & Y
\end{tikzcd}
$$
is derived Cartesian, then $h'^* \Ls$ is $f'$-ample.
\end{enumerate}
\end{prop}
\begin{proof}
\begin{enumerate}
\item The proof is a trivial consequence of Lemma \ref{RelAmpleLocalLem} and the fact that ample line bundles are relatively ample and stable under restrictions to open subschemes.

\item The proof is the same as the classical proof, see \cite{Stacks} Tags 0C4K (1) and 0892.

\item Without loss of generality we can assume that both $Y$ and $Y'$ are affine. It follows that $h'$ is affine, and as ample line bundles are stable under affine pullbacks, the claim follows. \qedhere
\end{enumerate}
\end{proof}

We also record the following useful fact. 

\begin{prop}\label{AmpleFamilyInQuasiProjProp}
Let $f: X \to Y$ be a morphism of Noetherian derived schemes having a relatively ample line bundle. If $Y$ has an ample family of line bundles, then so does $X$.
\end{prop}
\begin{proof}
Indeed, if $\Ls$ is a relatively ample line bundle, and $s$ is a global section of a line bundle $\Ms$ on $Y$ with $Y_s$ affine, then by Lemma \ref{LocLem} and the definition of relatively ample bundles, we conclude that we can find an affine open cover of $f^{-1} Y_s$ using global sections of $f^*(\Ms)^{\otimes n} \otimes \Ls^{\otimes m}$ for $n,m \gg 0$. Therefore, if $Y$ has an ample family, then so does $X$.
\end{proof}

For completeness' sake,  we also prove that relative ampleness can be checked on the truncation, at least for proper morphisms. 

\begin{prop}\label{AmpleCharProp2}
Let $A$ be a Noetherian simplicial ring, and let $X \to \Spec(A)$ be a proper morphism of derived schemes. Then, for a line bundle $\Ls$ on $X$, the following are equivalent:
\begin{enumerate}
\item $\Ls_\cl$ is ample on $X_\cl$;
\item given an almost perfect sheaf $\Fc$ on $X$ and $i \in \Zb$, we have that
$$\pi_i \big( \Gamma(X; \Fc \otimes \Ls^{\otimes n}) \big) \cong \Gamma_\cl(X; \pi_i(\Fc \otimes \Ls^{\otimes n}))$$
for $n \gg 0$;
\item $\Ls$ is ample.
\end{enumerate}
\end{prop}
\begin{proof}
Let us first assume that the truncation $\Ls_\cl$ is ample. By Grothendieck vanishing, it is enough to check that $2$ holds for $\Fc$ coherent. We will proceed by induction on $n$, where $n$ is the largest $i$ with $\pi_i(\Fc)$ nontrivial. Notice that the base case $n=0$ is easy: then
$$\Fc \otimes_{\Oc_X} \Ls^{\otimes m} \simeq \Fc \otimes_{\Oc_{X_\cl}} \Ls_\cl^{\otimes m},$$
and the claim follows from the classical result of Serre on the vanishing of sheaf cohomology (see, e.g., \cite{Stacks} Tag 0B5U). In general we have a cofibre sequence
$$\Fc' \to \Fc \to \pi_0(\Fc).$$
Applying the induction assumption to $\Fc'$ and $\pi_0(\Fc)$, the result holds for them for $m$ large enough. Then for any such $m$, it follows from the associated cohomology long exact sequence that
$$\pi_i\big(\Gamma(X; \Fc \otimes \Ls^{\otimes n})\big) \cong 
\begin{cases}
\pi_i\big(\Gamma(X; \Fc' \otimes \Ls^{\otimes n})\big) \cong \Gamma_\cl(X; \pi_i(\Fc \otimes \Ls^{\otimes n})) & \text{for $i > 0$;} \\
\Gamma_\cl(X; \pi_0(\Fc \otimes \Ls^{\otimes n}))\big) & \text{for $i = 0$;} \\
0 & \text{for $i < 0$.}
\end{cases}
$$
This shows that $\Ls$ satisfies condition 2.

Suppose then that $\Ls$ satisfies condition 2. Applying this to all truncated coherent sheaves, we can immediately conclude that $\Ls_\cl$ is ample using the result of Serre cited above. As global sections of $\Ls^{\otimes n}_\cl$ lift to global sections of $\Ls^{\otimes n}$ for $n \gg 0$, the ampleness of $\Ls_\cl$ implies that of $\Ls$.
\end{proof}

\subsubsection{Quasi-projective morphisms}\label{QProjSubSect}
 
In this section, we define and study the basic properties of quasi-projective morphisms. Like in the previous section, the results generalize some of those in the unpublished note \cite{annala-qpnote}. Let us begin with explaining what we mean by a quasi-projective morphism.

\begin{defn}\label{QProjDef}
A finite type morphism $f: X \to Y$ of Noetherian derived schemes is \emph{quasi-projective} if it admits a relatively ample line bundle. A proper quasi-projective morphism is called \emph{projective}.
\end{defn}

We have chosen to use the above definition for quasi-projectivity as it has the following good properties.

\begin{prop}\label{StabilityOfQProjProp}
Quasi-projective and projective morphisms between Noetherian derived schemes are stable under compositions and derived pullbacks.
\end{prop}
\begin{proof}
This follows immediately from Proposition \ref{AmpleLineBundlesInCompositionsLem}.
\end{proof}

Besides the good properties verified by the above result, projective morphisms have at least one additional benefit over the general proper morphisms: they often admit nice global factorizations. 

\begin{thm}\label{ExistenceOfFactorizationsThm}
Let $f: X \to Y$ be a projective morphism of Noetherian derived schemes.
\begin{enumerate}
\item If $Y$ has the resolution property, then there exists a vector bundle $E$ on $Y$ and a factorization of $f$ as 
$$X \stackrel i \hook \Pb(E) \stackrel \pi \to Y,$$
where $\pi$ is the canonical projection and $i$ is a closed embedding.

\item If $Y$ has an ample family of line bundles, then there exists line bundles $\Ls_1, ..., \Ls_r$ on $Y$ and a factorization of $f$ as 
$$X \stackrel i \hook \Pb(\Ls_1 \oplus \cdots \oplus \Ls_r) \stackrel \pi \to Y,$$
where $\pi$ is the canonical projection and $i$ is a closed embedding.

\item If $Y$ has an ample line bundle, then it is possible to find a similar factorization of form 
$$X \stackrel i \hook \Pb^n_Y \stackrel \pi \to Y.$$
\end{enumerate}
\end{thm}
\begin{rem}
We recall that a derived scheme $X$ is said to have the \emph{resolution property} if every connective coherent sheaf $\Fc$ admits a surjection $E \to \Fc$ (meaning that this induces a surjection on $\pi_0$) with $E$ a vector bundle on $X$. Every divisorial derived scheme has the resolution property by Proposition \ref{AmpleFamCharProp}.
\end{rem}
\begin{proof}
Let $\Ls$ be an $f$-ample line bundle on $X$. Then, by Proposition \ref{AmpleCharProp2}, we have that $f_*(\Ls^{\otimes n})$ is connective and that the natural morphism
$$f_*(\Ls^{\otimes n}) \to f_* \big( \pi_0(\Ls^{\otimes n}) \big)$$
identifies the right hand side as the $\pi_0$ of the left hand side for $n \gg 0$.  In particular, the $\pi_0$ of the natural map $f^*f_* (\Ls^{\otimes n}) \to \Ls^{\otimes n}$ can be identified with the classical morphism $f_\cl^* f_{\cl *} (\Ls_\cl^{\otimes n}) \to \Ls_\cl^{\otimes n}$. On the other hand, by \cite{Stacks} Tags 01VR and 01VU, the classical map is a surjection of sheaves and the induced morphism
$$X \hookrightarrow \mathbf{P}\big(f_{*} (\Ls^{\otimes n}) \big)$$
is a closed embedding over $Y$ for $n \gg 0$ (since the truncation is a closed embedding). As $f_* \Ls$ is coherent and connective, and as $Y$ has the resolution property, we can choose a surjection $E \to f_* \Ls$ from a vector bundle $E$, and the composition
$$i: X \hook \mathbf{P}\big(f_{*} (\Ls^{\otimes n}) \big) \hook \mathbf{P}(E) = \Pb(E^\vee)$$
is a closed embedding of $Y$-schemes, proving the first claim.

If $Y$ has an ample family of line bundles, then we may choose $E$ to be a sum of line bundles, proving the second claim. In the special case where $Y$ has an ample line bundle $\Ls$, we can choose $E$ to be a direct sum of copies of $\Ls^{\otimes -n}$ for some $n \geq 0$, proving the third claim.
\end{proof}

\subsubsection{Derived regular embeddings and quasi-smooth morphisms}

\begin{defn}\label{RegEmbDef}
Let $i: X \hookrightarrow Y$ be a closed embedding of derived schemes. Then it is called a \emph{derived regular embedding} of \emph{codimension $r$} if, Zariski locally on the target, it is equivalent to
$$\Spec\big(A \modmod (a_1, ..., a_r)\big) \hookrightarrow \Spec(A),$$
where $\modmod$ denotes the derived quotient ring construction (see \cite{khan-rydh} 2.3.1). In other words, $Z$ locally the derived vanishing locus of $r$ regular functions on $Y$.
\end{defn}

The following differential characterization is essentially just Proposition 2.3.8 of \cite{khan-rydh}.

\begin{prop}\label{ModifiedKhanRydh}
Let $f: X \hook Y$ be a closed embedding of Noetherian derived schemes with $X$ connected. Then the following are equivalent
\begin{enumerate}
\item $f$ is a derived regular embedding;
\item the shifted cotangent complex $\Lb_{X/Y}[-1]$ is a vector bundle;
\item the shifted cotangent complex $\Lb_{X/Y}[-1]$ has Tor-amplitude 0. 
\end{enumerate}
Henceforward we will denote $\Nc^\vee_{X/Y} := \Lb_{X/Y}[-1]$ and call it the \emph{conormal bundle}.
\end{prop}
\begin{proof}
The proof of \emph{loc. cit.} shows the equivalence of 1 and 2, but we have to pay some attention to one key step. Zariski locally on $Y$ the morphism $f$ is represented by a morphism $A \to B$ of simplicial commutative rings inducing a surjection on $\pi_0$. Let $F$ denote the homotopy fibre of this morphism considered as an $A$-module. The argument of \emph{loc. cit.} goes through once we show that $\pi_0(F)$ is finitely generated over $\pi_0(A)$ (otherwise one can not use Nakayama). But this is easy: we have an exact sequence
$$\pi_1(B) \to \pi_0(F) \to \pi_0(A) \to \pi_0(B) \to 0$$
and as $\pi_1(B)$ is a finitely generated $\pi_0(A)$-module, the claim follows.

Notice that 3 is equivalent to $\Lb_{X/Y}[-1]$ being flat. Since, we have already shown that $\pi_0(\Lb_{X/Y}[-1])$ is coherent, we conclude that 3 implies 2. Since vector bundles are flat, we are done. 
\end{proof}

\begin{defn}\label{StrSmQSmDef}
Let $f: X \to Y$ be a morphism of derived schemes. Then
\begin{enumerate}
\item $f$ is called \emph{strong} if the natural maps $f^*_\cl \pi_i(\Oc_Y) \to \pi_i(\Oc_X)$ are isomorphisms for all $i \geq 0$;
\item $f$ is called \emph{smooth} if it is strong and the truncation $f_\cl$ is smooth in the classical sense;
\item $f$ is called \emph{quasi-smooth} if, Zariski locally on $X$, it admits a factorization
$$X \stackrel i \hook P \stackrel p \to Y$$
with $i$ a derived regular embedding and $p$ a smooth morphism.
\end{enumerate}
\end{defn}

The following alternative characterization of smooth morphisms will be useful for us.

\begin{lem}\label{SmCharLem}
A morphism $f: X \to Y$ of Noetherian derived schemes is smooth if and only if $f_\cl$ is of finite type and $\Lb_{X/Y}$ is a vector bundle.
\end{lem}
\begin{proof}
This is well known if $X$ and $Y$ are classical (finite type morphisms of Noetherian schemes are of finite presentation). To prove the equivalence in general, consider the derived Cartesian square
$$
\begin{tikzcd}
X' \arrow[]{r}{f'} \arrow[]{d}{\iota_X} & Y_\cl \arrow[]{d}{\iota_Y} \\
X \arrow[]{r}{f} & Y.
\end{tikzcd}
$$
If $f$ is smooth, then it is strong and flat, and therefore $X'$ is classical and identified with $X_\cl$, which proves that $\iota_X^* \Lb_{X/Y}$, and hence $\Lb_{X/Y}$, is a vector bundle. To prove the converse implication, it is enough to show that if $\Lb_{X/Y}$ is a vector bundle and $Y$ is classical, then $X$ is classical as well. Locally we have
$$A \to A [x_1,...,x_n] \to B := A [x_1,...,x_n] \modmod (f_1,...,f_r)$$
and the assumptions imply that we have a split injection of $\pi_0(B)$-modules
$$\langle df_1,...,df_r \rangle_{\pi_0(B)} \hook \pi_0(B) \otimes_{A[x_1,...,x_n]} \Omega_{A[x_1,...,x_n] / A}.$$
But this implies that 
$$A \to A [x_1,...,x_n] \to B := A [x_1,...,x_n] / (f_1,...,f_r)$$
is a relative global complete intersection (in the sense of \cite{Stacks} Tag 00SP) with smooth fibres (differential criterion for smoothness). Hence $f_1,..,f_r$ is a regular sequence around the points of $V(f_1,...,f_r)$ and $B \simeq A [x_1,...,x_n] / (f_1,...,f_r)$ is discrete.
\end{proof}

The following result is then an easy consequence of Proposition \ref{ModifiedKhanRydh}. 

\begin{prop}\label{QSmCharProp}
Let $f: X \to Y$ be a morphism of Noetherian derived schemes. Then the following are equivalent:
\begin{enumerate}
\item $f$ is quasi-smooth;
\item the truncation $f_\cl$ is of finite type and the cotangent complex $\Lb_{X/Y}$ is a perfect complex of Tor-dimension 1;
\item the truncation $f_\cl$ is of finite type and the cotangent complex $\Lb_{X/Y}$ has Tor-dimension 1.
\end{enumerate}
\end{prop}
\begin{proof}
Since the truncation is of finite type, $f$ admits a local factorization
$$X \stackrel i \hook P \stackrel p \to Y$$
with $p$ smooth (take $P = \Ab^n \times Y$ for $n$ large enough). The claim then follows by investigating the fundamental triangle
$$i^* \Lb_{P / Y} \to \Lb_{X/Y} \to \Lb_{X/P}$$
and using Proposition \ref{ModifiedKhanRydh} and Lemma \ref{SmCharLem}.
\end{proof}

If $f: X \to Y$ is quasi-smooth and $\Lb_{X/Y}$ has constant virtual rank $d$ at the points of $x$, then we say that $f$ has \emph{relative virtual dimension $d$}. It is then easy to prove the following basic result.

\begin{prop}
Quasi-smooth morphisms of Noetherian derived schemes are stable under compositions and derived pullbacks. Moreover, relative virtual dimension is additive under composition, and is preserved under derived pullback. \qed
\end{prop}

\subsubsection{Derived complete intersection rings}\label{DerComplIntSubSect}

The purpose of this section is to introduce an absolute version of quasi-smoothness, and verify it has desirable properties. A simplicial commutative ring is called \emph{local} if its truncation is local.

\begin{defn}
A local Noetherian simplicial ring $A$ is called \emph{derived complete intersection} if only finitely many $\pi_i(A)$ are nontrivial and the cotangent complex $\Lb_{\kappa / A}$, where $\kappa$ is the residue field of $A$, is perfect and concentrated in degrees less than or equal to 2.

A Noetherian derived scheme is called \emph{derived complete intersection} if its local rings are derived complete intersection rings. We also call such derived schemes \emph{derived regular}.
\end{defn}

\begin{ex}
If $A$ is a regular local ring and $a_1,...,a_r$ are elements of $A$, then the derived quotient ring $A \modmod (a_1, ..., a_r)$ is a derived complete intersection ring.
\end{ex}

The following result explains what we want to capture with this definition.

\begin{prop}\label{AbsoluteQSmProp}
Suppose $B$ is a derived complete intersection local ring, and let $f: A \to B$ be a morphism from a regular local ring $A$ which is a surjection on $\pi_0$. Then there exist elements $a_1,..., a_r$ and an equivalence 
$$B \simeq A \modmod (a_1,...,a_r)$$
of $A$-algebras. In other words, a closed immersion from a derived regular scheme to a regular scheme is a derived regular embedding.
\end{prop}
\begin{proof}
Let $\kappa$ be the residue field of $A$, and consider the sequence $A \to B \to \kappa$ inducing the fundamental triangle
$$\kappa \otimes_B^L \Lb_{B/A} \to \Lb_{\kappa / A} \to \Lb_{\kappa / B}.$$
As $A$ is regular, $\Lb_{\kappa / A}$ is a $\kappa$-vector space in degree 1, and $B$ is a derived complete intersection, it follows that $\kappa \otimes_B^L \Lb_{B/A}$ is concentrated in degree 1, and the claim follows from Proposition \ref{ModifiedKhanRydh}.
\end{proof}

The following alternative characterization of derived complete intersection schemes is also useful.

\begin{prop}\label{DCompIntAltCharProp}
Let $X$ be a Noetherian derived scheme. Then the following are equivalent:
\begin{enumerate}
\item $X$ is a derived complete intersection scheme;
\item the absolute cotangent complex $\Lb_X$ has Tor-dimension 1;
\item for all maps $X \to Y$ with $Y$ a regular Noetherian scheme, $\Lb_{X/A}$ has Tor-dimension 1.
\end{enumerate}
\end{prop}
\begin{proof}
Let $B$ be a derived local ring of $X$ with residue field $\kappa$, let $A$ be a regular ring and consider the sequence $A \to B \to \kappa$ inducing the cofibre sequence
$$\kappa \otimes_B \Lb_{B/A} \to \Lb_{\kappa/A} \to \Lb_{\kappa/B}.$$
Since $\Lb_{\kappa/A}$ has Tor dimension 1, we conclude that $\Lb_{\kappa/B}$ has Tor-dimension 2 if and only if $\Lb_{B/A}$ has Tor-dimension 1, proving the equivalence of the three conditions.
\end{proof}

Finally we need to be able to talk about the dimension of a derived complete intersection scheme.

\begin{defn}\label{VDimDef}
Let $A$ be a derived complete intersection local ring with residue field $\kappa$. Then the \emph{virtual dimension} of $A$ is defined as
$$\dim(A) := \dim_\kappa \pi_1(\Lb_{\kappa / A}) - \dim_\kappa \pi_2(\Lb_{\kappa / A}).$$
If $X$ is a derived complete intersection scheme and $x \in X$ is a point, then we define the \emph{virtual dimension of $X$ at $x$}, $\dim_x(X)$, as the virtual dimension of the derived local ring $\Oc_{X,x}$. Finally, the \emph{virtual dimension} of such an $X$, $\dim(X)$, is defined as the supremum of $\dim_x(X)$, where $x$ ranges over all the points of $X$.
\end{defn}

\begin{rem}
Notice that if $A$ is a classical complete intersection ring, then the virtual dimension of $A$ agrees with the Krull dimension of $A$. Indeed, for regular local rings this is easy, and for more general complete intersection local rings it follows from the stability of cotangent complex in derived pullbacks that
$$\dim(A) = \dim(\hat A),$$
where $\hat A$ denotes the completion of $A$ with respect to its maximal ideal. On the other hand, since $\hat A$ is a quotient of a regular local ring by a regular sequence, one observes easily that the virtual dimension of $\hat A$ agrees with its Krull dimension. Since also the Krull dimensions of $A$ and $\hat A$ are the same, we conclude that the virtual dimension of $A$ agrees with its Krull dimension. 
\end{rem}

\begin{rem}\label{DimProblemRem}
One would expect that, at least under suitable hypotheses, a connected derived complete intersection scheme is equicodimensional in the sense that $\dim_x(X) = \dim(X)$ for all closed points $x$ of $X$. It is not hard to see that this is true for $X$ finite type over a field $k$, but we won't try to pursue this in any greater generality (doing so would take us too far afield). This means that we don't know if virtual dimension is a good notion: for instance, consider a derived regular embedding $Z \hook X$ of codimension $r$, and suppose $X$ is a derived complete intersection scheme of virtual dimension $d$. Although $Z$ is clearly a derived complete intersection scheme, it is not clear at all that it has virtual dimension $d-r$ (and in fact we would expect this to be false in general). 

This potential for ill behavior will have ramifications later in this article. Indeed, we will construct base independent homological bordism groups $\Omega_\bullet(X)$ but we won't be able to prove that these groups admit a grading by the virtual dimension of cycles.
\end{rem}

\subsubsection{Derived blow ups}

One of the main technical tools we are going to need in this article is the construction of derived blow ups and derived deformation to normal cone from \cite{khan-rydh}. Let us recall the definitions and the results we are going to use:

\begin{defn}\label{DivisorOverZDef}
Let $Z \hookrightarrow X$ be a derived regular embedding. Then, for any $X$-scheme $S$, a \emph{virtual Cartier divisor} on $S$ \emph{lying over $Z$} is the datum of a commutative diagram
\begin{center}
\begin{tikzcd}
D \arrow[hook]{r}{i_D} \arrow[]{d}{g} & S \arrow[]{d}{} \\
Z \arrow[hook]{r}{} & X
\end{tikzcd}
\end{center}
such that
\begin{enumerate}
\item $i_D$ is a quasi-smooth closed embedding of codimension 1;
\item the truncation is a Cartesian square;
\item the canonical morphism
$$g^* \Nc^\vee_{Z/X} \to \Nc^\vee_{D/S}$$
induces a surjection on $\pi_0$.
\end{enumerate}
\end{defn}

It is now possible to define derived blow ups via a universal property.

\begin{defn}\label{DerivedBlowUpDef}
Let $Z \hookrightarrow X$ be a derived regular embedding. Then the \emph{derived blow up} $\bl_Z(X)$ is the $X$-scheme representing virtual Cartier divisors lying over $Z$. In other words, given an $X$-scheme $S$, the space of $X$-morphisms 
$$S \to \bl_Z(X)$$
is naturally identified with the maximal sub $\infty$-groupoid of the $\infty$-category of virtual Cartier divisors of $S$ that lie over $Z$.
\end{defn}

\begin{thm}\label{BlowUpPropertiesThm}
Let $i: Z \hookrightarrow X$ be a derived regular embedding with $X$ Noetherian. Then
\begin{enumerate}
\item the derived blow up $\bl_Z(X)$ exists and is unique up to contractible space of choices;

\item the structure morphism $\pi: \bl_Z(X) \to X$ is projective, quasi-smooth, and induces an equivalence 
$$\bl_Z(X) - \Ec \to X - Z,$$
where $\Ec$ is the universal virtual Cartier divisor on $\bl_Z(X)$ lying over $Z$ (also called the \emph{exceptional divisor});

\item the derived blow up $\bl_Z(X) \to X$ is stable under derived base change;

\item the exceptional divisor $\Ec$ is naturally identified with $\Pb_Z(\Nc_{Z/X})$;

\item if $Z \stackrel i \hookrightarrow X \stackrel j \hookrightarrow Y$ is a sequence of quasi-smooth closed embeddings, then there is a natural derived regular embedding $\tilde j: \bl_Z(X) \hookrightarrow \bl_Z(Y)$ called the \emph{strict transform};

\item given derived regular embeddings $i: Z \hook X$ and $j: Y \hook X$, the strict transforms $\tilde i$ and $\tilde j$ do not meet in $\bl_{Z \cap Y}(X)$;

\item if $Z$ and $X$ are classical schemes (so that $Z \hookrightarrow X$ is lci), there is a natural equivalence 
$$\bl_Z(X) \simeq \bl_{Z}^\mathrm{cl}(X),$$
where the right hand side is the classical blow up.
\end{enumerate}
\end{thm}
\begin{proof}
The statements 1, 3, 4, 5 and 7 are directly from \cite{khan-rydh} Theorem 4.1.5. The second claim is otherwise from \emph{loc. cit.}, except that we claim $\pi$ to be projective and not just proper, which follows from the fact that $\Oc(-\Ec)$ is $\pi$-ample. For a proof of 6, see for example \cite{annala-cob} Lemma 4.5.
\end{proof}

We will also make use of the following results. 

\begin{prop}\label{BlowUpOfVectBundleSectionProp}
Let $X$ be a derived scheme, $E$ a vector bundle on $X$, $s$ a global section of $E$ and $Z$ the derived vanishing locus of $s$ on $X$. Then the natural short exact sequence
$$0 \to \Oc(-1) \to E \to Q \to 0$$
induces a natural equivalence 
$$\abs{\Gamma(X; E)} \simeq \abs{\Gamma(\Pb(E); Q)}$$
of spaces of global sections sending $s$ to $\tilde s$, and the derived vanishing locus of $\tilde s$ inside $\Pb(E)$ is equivalent to $\bl_Z(X)$ as an $X$-scheme. Moreover, $\Oc(-1)$ restricts to $\Oc(\Ec)$ on $\bl_Z(X)$.
\end{prop}
\begin{proof}
This is \cite{annala-chern} Proposition 2.7.
\end{proof}

\begin{prop}\label{BlowUpIsResolutionSchemeProp}
Let $X$ be a derived scheme, $E$ a vector bundle on $X$, $s$ a global section of $E$ and $Z$ the derived vanishing locus on $X$. Then $s: \Oc \to E$ factors through the natural injection
$$\Oc(\Ec) \hook E$$
of vector bundles on $\bl_Z(X)$.
\end{prop}
\begin{proof}
Indeed, since the lower row of the diagram
$$
\begin{tikzcd}
 & \Oc \arrow[dashed]{ld}{} \arrow[]{d}{s} \arrow[]{rd}{\tilde s} & \\
\Oc(-1) \arrow[hook]{r} & E \arrow[twoheadrightarrow]{r} & Q
\end{tikzcd}
$$
is a cofibre sequence, a homotopy $\tilde s \sim 0$ induces in a natural way the dashed arrow, and therefore the result follows from Proposition \ref{BlowUpOfVectBundleSectionProp}.
\end{proof}

\subsection{Bivariant algebraic cobordism over $A$}

Let $A$ be a Noetherian ring of finite Krull dimension. The purpose of this section is to recall the bivariant algebraic $A$-cobordism from \cite{annala-pre-and-cob}. Let us recall the following definition.

\begin{defn}\label{FunctorialityDef}
A \emph{functoriality} is a tuple $\Fs = (\Cc, \Cs, \Is, \Ss)$, where 
\begin{enumerate}
\item $\Cc$ is an $\infty$-category with a (distinguished) final object $pt$ and all fibre products;

\item $\Cs$ is a class of morphisms in $\Cc$ called \emph{confined morphisms} which contains all equivalences and is closed under compositions, pullbacks and homotopy equivalences of morphisms;

\item $\Is$ is a class of Cartesian squares in $\Cc$ called \emph{independent squares} which contains all squares of form
\begin{center}
\begin{tikzcd}
X \arrow[]{r}{f} \arrow[]{d}{\mathrm{Id}_X} & Y \arrow[]{d}{\mathrm{Id}_Y} \\
X \arrow[]{r}{f} & Y
\end{tikzcd}
\ \ and \ \ \
\begin{tikzcd}
X \arrow[]{r}{\mathrm{Id}_X} \arrow[]{d}{f} & X \arrow[]{d}{f} \\
Y \arrow[]{r}{\mathrm{Id}_Y} & Y,
\end{tikzcd}
\end{center}
and is closed under vertical and horizontal compositions in the obvious sense as well as equivalences of Cartesian squares

\item $\Ss$ is a class of morphisms in $\Cc$ called \emph{specialized morphisms} which contains all equivalences and is closed under compositions and homotopy equivalences of morphisms.
\end{enumerate}
\end{defn}

Given a functoriality $\Fs$, one defines the notion of a bivariant theory with functoriality $\Fs$ as in \cite{annala-pre-and-cob} Section 2.1. Let us recall what it means for a bivariant theory to be additive.

\begin{defn}\label{AdditivityDef}
Suppose $\Fs = (\Cc, \Cs, \Is, \Ss)$ is such a functoriality that $\Cc$ has all finite coproducts and the canonical inclusions are confined. We say that a bivariant theory $\Bb^*$ with functoriality $\Fs$ is \emph{additive} if the morphisms 
$$\Bb^*(X_1 \to Y) \oplus \Bb^*(X_2 \to Y) \xrightarrow{\iota_{1*} + \iota_{2*}} \Bb^*\Bigr(X_1 \coprod X_2 \to Y\Bigr)$$
are isomorphisms for all $X_1, X_2$ and $Y$, where $\iota_i$ is the canonical inclusion $X_i \to X_1 \coprod X_2$.
\end{defn}

We will denote by $\Fs_A$ the functoriality $(\Cc_A, \Cs_A, \Is_A, \Ss_A)$, where $\Cc_A$ is the $\infty$-category of quasi-projective derived $A$-schemes, $\Cs_A$ consists of the projective morphisms, $\Is_A$ consists of all derived Cartesian squares and $\Ss_A$ consists of quasi-smooth morphisms. Let us start by recalling the universal $A$-precobordism.

\begin{defn}\label{UnivAPrecobDef}
The \emph{universal $A$-precobordism theory} $\PCob^*_A$ is the additive stably oriented bivariant theory with functoriality $\Fs_A$ defined as follows: given a homotopy class of morphisms $X \to Y$ in $\Cc$, the Abelian group $\PCob^d_A(X \to Y)$ is generated by symbols $[V \to X]$ of equivalence classes of such projective morphisms $V \to X$ that the composition $V \to Y$ is quasi-smooth and of relative virtual dimension $-d$. These symbols are subjected to two sets of relations. First of all, taking disjoint unions is identified with summation (this enforces additivity). Secondly, given a projective morphism $W \to \Pb^1 \times X$ so that the composition $W \to \Pb^1 \times Y$ is quasi-smooth and of relative virtual dimension $-d$, and supposing that the fibre $W_\infty$ of $W \to \Pb^1$ over $\infty$ is a sum of virtual Cartier divisors $D_1$ and $D_2$, then
$$[W_0 \to X] = [D_1 \to X] + [D_1 \to X] - [\Pb_{D_1 \cap D_2}(\Oc(D_1) \oplus \Oc) \to X] \in \PCob^d_A(X \to Y),$$
where $W_0$ is the fibre of $W \to \Pb^1$ over $0$, and where $\cap$ denotes the derived intersection.

Bivariant pushforward and pullback are defined in the obvious way using compositions of morphisms and derived pullbacks. The bivariant product is defined by bilinearly extending the following formula: given $[V \to X] \in \PCob^a_A(X \to Y)$ and $[W \to Y] \in \PCob^b_A(Y \to Z)$, form the derived Cartesian square
$$
\begin{tikzcd}
V' \arrow[]{d} \arrow[]{r} & X' \arrow[]{d} \arrow[]{r} & W \arrow[]{d} & \\
V \arrow[]{r} & X \arrow[]{r} & Y \arrow[]{r} & Z,
\end{tikzcd}
$$
and define 
$$[V \to X] \bullet [W \to Y] := [V' \to X] \in \PCob^{a + b}_A(X \to Z).$$
The orientation is given as follows: if $X \to Y$ is a quasi-smooth morphism, then $1_{X/Y} := [X \to X] \in \PCob^*(X \to Y)$.

We will denote by $\PCob^*_A(X)$ the universal $A$-precobordism rings $\PCob^*_A(X \to X)$. Notice that the bivariant product has a simpler formula in this case: 
$$[V \to X] \bullet [W \to X] = [V \times_X W \to X] \in \PCob^*_A(X).$$
We will use shorthand notation for the unit $1_X := 1_{X/X} \in \PCob^*(X)$. Let us also denote by $\PCob^A_*(X)$ the universal $A$-prebordism groups $\PCob^{-*}_A(X \to pt)$. Note that the bivariant product gives $\PCob^A_*(X)$ the structure of an $\PCob^*_A(X)$-module.
\end{defn}

Recall that given a line bundle $\Ls$ on $X \in \Cc_A$, we define the \emph{Euler class} (usually called the \emph{first Chern class}, but we refrain from using this terminology until we have constructed Chern classes in Section \ref{GenChernClassSubSect}) of $\Ls$ as
$$e(\Ls) := [Z_s \hook X] \in \PCob^*_A(X),$$
where $s$ is a global section of $\Ls$, and $Z_s$ is the derived vanishing locus of $s$. We recall that by \cite{annala-chern} Theorem 3.4 there is a formal group law 
$$F_A(x,y) \in \PCob^*_A(pt)[[x,y]]$$
such that for all line bundles $\Ls_1$ and $\Ls_2$ on $X \in \Cc_A$, the equality
$$e(\Ls_1 \otimes \Ls_2) = F_A\big(e(\Ls_1), e(\Ls_2)\big) \in \PCob^*(X)$$
holds. Following \cite{annala-pre-and-cob}, we will define bivariant algebraic $A$-cobordism as a quotient of $\PCob^*_A$. However, before doing so, we have to recall some terminology.

Suppose we are given the data of a virtual Cartier divisor $D \hook W$, and an equivalence
$$D \simeq n_1 D_1 + \cdots + n_r D_r$$
of virtual Cartier divisors on $W$ with $n_i > 0$. Then, denoting by $+_{F_A}$ the formal addition given by the formal group law $F_A$ and by $[n]_{F_A} \cdot$ the formal multiplication (iterated formal addition), the formal power series 
$$[n_1]_{F_A} \cdot x_1 +_{F_A} \cdots +_{F_A} [n_r]_{F_A} \cdot x_r$$ 
in $r$ variables has a unique expression of form 
$$\sum_{I \subset \{1,2,...,r \}} \textbf{x}^I F^{n_1,...,n_r}_{A,I}(x_1,...,x_r),$$
where 
$$\textbf{x}^I = \prod_{i \in I} x_i$$
and $F_{A, I}^{n_1,...,n_r}(x_1,...,x_r)$ contains only variables $x_i$ such that $i \in I$. Note that $F_{A, \emptyset}^{n_1,...,n_r}(x_1,...,x_r) = 0$. Using this notation, we make the following definition.

\begin{defn}\label{ZetaClassDef}
Let everything be as above. We define
$$
\zeta_{W,D,D_1,...,D_r} := \sum_{I \subset \{1,2,...,r\}} \iota^I_*\Biggl(F^{n_1,...,n_r}_{A,I} \Bigl( e\bigl(\Oc(D_1)\bigr), ...,e\bigl(\Oc(D_r)\bigr) \Bigr) \bullet 1_{D_I/pt} \Biggr) \in \PCob^A_*(D),
$$
where $\iota^I$ is the inclusion $D_I \hookrightarrow D$ of the derived intersection $\cap_{i \in I} D_i$ inside $W$.
\end{defn}

In the most important special case, the composition is provided by the data of an $A$-snc divisor.

\begin{defn}\label{ASNCDivDef}
Let $W$ be a smooth quasi-projective $A$-scheme. Then an \emph{$A$-snc divisor} on $W$ is the data of an effective Cartier divisor $D \hookrightarrow W$, Zariski connected effective Cartier divisors $D_1,...,D_r \hookrightarrow W$ with 
$$D_I := \cap_{i \in I} D_i$$
smooth and of the expected relative dimension over $A$ for all $I \subset \{1,2,...,r\}$, so that there exist positive integers $n_1,...,n_r$ such that
$$D \simeq n_1 D_1 + \cdots + n_r D_r$$
as effective Cartier divisors on $W$. Somewhat abusively, we will use the shorthand notation
$$\zeta_{W,D} := \zeta_{W,D, D_1, ..., D_r}$$
whenever $D \hook W$ and $D_i$ form an $A$-snc divisor. 
\end{defn}

We are now ready to recall the definition of bivariant algebraic cobordism over $A$.

\begin{defn}\label{BivACobDef}
\emph{Bivariant algebraic $A$-cobordism theory} is the additive stably oriented bivariant theory with functoriality $\Fs_A$ defined as the quotient
$$\Omega^*_A := \PCob^*_A / \langle \Rc^\mathrm{snc}_A \rangle,$$
where $\Rc^\mathrm{snc}_A$ is the bivariant subset of \emph{snc relations} consisting of 
$$\zeta_{W,D} - \theta(\pi_D) \in \PCob^*_A(D \to pt),$$ 
where $D \hook W$ ranges over all $A$-snc divisors, and $\langle \Rc^\mathrm{snc}_A \rangle$ is the bivariant ideal generated by $\Rc^\mathrm{snc}_A$.
\end{defn}

We recall that $\PCob^*_A$ and $\Omega^*_A$ were proven to have many good properties in \cite{annala-cob, annala-yokura, annala-chern, annala-pre-and-cob}: if $k$ is a field of characteristic 0, the associated homology theory of $\Omega^*_k$ recovers Levine--Morel's algebraic bordism, the Grothendieck ring of vector bundles can be recovered from $\PCob^*_A$ analogously to Conner--Floyd theorem, projective bundle formula holds for $\PCob^*_A$, and $\PCob^*_A$ can be used to construct a candidate for Chow-cohomology theory satisfying a generalization of Grothendieck--Riemann--Roch theorem. We will not recall the statements here, as we are going to give more general versions in Section \ref{GenChernClassSubSect}.

\section{Bivariant algebraic cobordism}\label{BivAlgCobSect}

The purpose of this section is to prove that bivariant algebraic $A$-cobordism groups are independent from the base ring $A$, and to define these groups more generally for morphisms of finite dimensional divisorial Noetherian derived schemes. In particular, we will not need to assume quasi-projectivity over a Noetherian base ring $A$ anymore. 

The structure of this section is as follows: in Section \ref{BaseIndSubSect} we prove that $\Omega^*_A(X \to Y)$ is independent from the base ring $A$ by giving an alternative characterization for the bivariant ideal of $\PCob^*_A$ generated by the snc relations. Using a slightly modified version of this alternative characterization, we define in Section \ref{BivCobConsSubsect} bivariant theories $\PCob^\bullet$ and $\Omega^\bullet$ on the $\infty$-category of divisorial and finite dimensional Noetherian derived schemes which generalize the theories $\PCob^*_A$ and $\Omega^*_A$ respectively. In Section \ref{NilpEulerClassSubSect} we analyze the potential failure of nilpotence for Euler classes of vector bundles on $X \in \Cc_a$, and define modified theories $\PCob'^\bullet$ and $\Omega'^\bullet$ where Euler classes of vector bundles are provably nilpotent. Note that nilpotence of Euler classes is a fundamental ingredient in the proofs of Section  \ref{GeneralPBFSect}, explaining why we go to such lengths to make sure that this property holds.

\subsection{Base independence of bivariant algebraic $A$-cobordism}\label{BaseIndSubSect}

The purpose of this section is to prove that if $X \to Y$ is a morphism of derived schemes with $X$ and $Y$ quasi-projective over Noetherian rings $A$ and $B$ of finite Krull dimension, then $\Omega^*_A(X \to Y)$ and $\Omega^*_B(X \to Y)$ are canonically isomorphic. While proving this, we also obtain an explicit construction of $\Omega^*_A(X \to Y)$ as a quotient of $\PCob^*_A(X \to Y)$, without the need of considering either of these groups as parts of a bivariant theory.

We begin with the following trivial observation.

\begin{lem}
Suppose $X \to Y$ is a morphism of derived schemes with $X$ and $Y$ quasi-projective over Noetherian rings $A$ and $B$ of finite Krull dimension. Then the groups $\PCob^*_A(X \to Y)$ and $\PCob^*_B(X \to Y)$ are canonically isomorphic. \qed
\end{lem}

In order to obtain the desired result, we need to analyze the bivariant ideal $\langle \Rc^\mathrm{snc}_A \rangle$ used in Definition \ref{BivACobDef}. The main tool we are going to use are the following schemes.

\begin{cons}\label{VSchemeCons}
Let $n \geq 0$, and consider the scheme $\Pb^n \times \Pb^n \times \Pb^n \times \Pb^n$, where the projective spaces are over $\Spec(\Zb)$. Consider the vector bundle $E_n := \Oc(1,-1,0,0) \oplus \Oc(0,0,1,-1)$ over $(\Pb^n)^{\times 4}$, and notice that there are tautological sections
$$s_1: \Oc_{E_n} \to \Oc(1,-1,0,0)$$
and 
$$s_2: \Oc_{E_n} \to \Oc(0,0,1,-1)$$
on $E_n$, so that the (derived) vanishing locus of $s_1$ is $\Oc(0,0,1,-1)$ and the vanishing locus of $s_2$ is $\Oc(1,-1,0,0)$. We define $\Vc^n$ as the vanishing locus of $s_1 s_2$ inside $E_n$. We also define $\Vc^n_i$ as the vanishing loci of $s_i$, and $\Vc^n_{12}$ as the vanishing locus of $s_1, s_2$; notice that there are (non-regular) closed embeddings $\iota^i: \Vc^n_i \hook \Vc^n$  which give a closed cover of $\Vc^n$, and a (non-regular) closed embedding $\iota^{12}: \Vc^n_{12} \hook \Vc^n$. We will also denote
$$\Nc_1 := \Nc_{\Vc^n_{12} / \Vc^n_1} \cong \Oc(0,0,1,-1) \vert_{\Vc^n_{12}}$$
and
$$\Nc_2 := \Nc_{\Vc^n_{12} / \Vc^n_2} \cong \Oc(1,-1,0,0) \vert_{\Vc^n_{12}}.$$
If $Y$ is a derived scheme, we define $\Vc^n_Y, \Vc^n_{Y,i}$ and $\Vc^n_{Y,12}$ as $Y \times \Vc^n$, $Y \times \Vc^n_i$ and $Y \times \Vc^n_{12}$ respectively. 
\end{cons}

Before moving on, let us record the following easy observation.

\begin{lem}
Let $V$ and $Y$ be Noetherian derived schemes, and suppose $f: V \to \Vc^n_Y$ is a morphism. Then the following are equivalent
\begin{enumerate}
\item $f$ is quasi-smooth;
\item the maps $V_i := \Vc^n_{Y,i} \times_{\Vc^n_Y} V \to Y$ are quasi-smooth for $i=1,2$. 
\end{enumerate}
\end{lem}
\begin{proof}
Since $\Vc^n_i$ cover $\Vc^n$, it is easy to check (say, using Proposition \ref{QSmCharProp}) that $f$ being quasi-smooth is equivalent to $V_i \to \Vc^n_{Y,i}$ being quasi-smooth for $i=1,2$. Since $\Vc^n_{Y,i} \to Y$ are smooth, the claim follows.
\end{proof}

We then consider the following bivariant ideal.

\begin{defn}\label{SNCIdealAltDef}
Let $A$ be a Noetherian ring of finite Krull dimension, and let $X \to Y$ be a morphism in $\Cc_A$. We will define a bivariant ideal $\Ic_A \subset \PCob^*_A(X \to Y)$. Given a projective morphism $f: V \to X$ with the composition $V \to Y$ quasi-smooth and a quasi-smooth $Y$-morphism $V \to \Vc^n_Y$, we will denote by $V_i$ and $V_{12}$ the derived fibre products $\Vc^n_{Y,i} \times_{\Vc^n_Y} V$ and $\Vc^n_{Y,12} \times_{\Vc^n_Y} V$ respectively, and by $f^{12}$ the induced morphism $V_{12} \to X$. Given such data, consider the element
$$[V \to X] - [V_1 \to X] - [V_2 \to X] - f^{12}_* \Bigg( \sum_{i,j \geq 1} a_{ij} e(\Nc_{V_{12}/V_1})^{i-1} \bullet e(\Nc_{V_{12}/V_2})^{j-1} \bullet 1_{V_{12}/Y}\Bigg)$$
of $\PCob^*_A(X \to Y)$, where $a_{ij}$ are the coefficients of the formal group law $F_A$ of $\PCob^*_A$, and define $\Ic_A(X \to Y)$ as the subgroup of $\PCob^*_A(X \to Y)$ generated by all such elements. The fact that $\Ic_A$ is a bivariant ideal follows trivially from the fact that projective and quasi-smooth morphisms are stable under pullbacks and compositions.
\end{defn}

Our goal is to prove that $\Ic_A = \langle \Rc^\mathrm{snc}_A \rangle$. The following lemma shows that $\langle \Rc^\mathrm{snc}_A \rangle \subset \Ic_A$.

\begin{lem}\label{LSRelationsContainedInILem}
Let $A$ be a Noetherian ring of finite Krull dimension. Suppose $X$ is a quasi-smooth and quasi-projective derived $A$-scheme, and let $D \hook X$ be a virtual Cartier divisor. Then, if $D \simeq n_1 D_1 + \cdots + n_r D_r$ for some virtual Cartier divisors $D_i \hook X$ and $n_i > 0$, the equality
$$1_{D/pt} = \zeta_{X,D,D_1,...,D_r} \in \PCob^A_*(D) / \Ic_A(D \to pt)$$
holds, where $\zeta_{X,D,D_1,...,D_r}$ is as in Definition \ref{ZetaClassDef}.
\end{lem}
\begin{proof}
For simplicity we will denote by $x_i$ the Euler class $e(\Oc(D_i))$, by $+_{F_A}$ the formal addition and by $[n]_{F_A} \cdot$ the formal multiplication by $n$. Note that the claim is trivially true for $r=1$ and $n = n_1 = 1$. We will proceed by induction on $n$. If $D \simeq (n+1) D_1$, then $D$ is the derived vanishing locus of $s^{n+1}$, where $s \in \Oc(D_1)$. Since $X$ is quasi-projective over $A$, it is possible to find $d \geq 0$ and a morphism $X \to (\Pb^d)^{\times 4}$ such that $\Oc(1,-1,0,0)$ and $\Oc(0,0,1,-1)$ pull back to $\Oc(D_1)$ and $\Oc(n D_1)$ respectively. The global sections $s$ and $s^{n}$ then induce a morphism 
$$X \to E_n.$$
so that the tautological sections $s_1$ and $s_2$ (see Construction \ref{VSchemeCons}) pull back to $s$ and $s^{n}$ respectively. This also induces a quasi-smooth morphism $D \to \Vc^d_A$ with
$$
\begin{tikzcd}
D_1 \arrow[hook]{r}{\iota^1} \arrow[]{d} & D \arrow[]{d} & \arrow[hook']{l}[swap]{\iota^2} n D_1 \arrow[]{d} \\
\Vc^d_{A,1} \arrow[hook]{r} & \Vc^d_A & \arrow[hook']{l} \Vc^d_{A,2}
\end{tikzcd}
$$
homotopy Cartesian. In particular 
\begin{align*}
[D \to D] &= [D_1 \to D] + [n D_1 \to D] \\
&+ \iota^{12}_* \Bigg( \sum_{i,j \geq 1} a_{ij} e\big(\Oc(D_1)\big)^{i-1} \bullet e\big(\Oc(n D_1)\big)^{j-1} \bullet 1_{D_1 \cap n D_1 / pt} \Bigg)
\end{align*}
in $\Omega^*\big(D \to \Spec(A)\big)$, where $\iota^{12}$ is the embedding $D_1 \cap (n-1) D_1 \hook D$. Applying the inductive assumption, we compute that
$$[n D_1 \to D] = \iota^1_*\big(F^n_{A,\{1\}} (x_1) \bullet 1_{D_1 / pt}\big)$$
and
\begin{align*}
& \iota^{12}_* \Bigg( \sum_{i,j \geq 1} a_{ij} e\big(\Oc(D_1)\big)^{i-1} \bullet e\big(\Oc(n D_1)\big)^{j-1} \bullet 1_{D_1 \cap n D_1 / pt} \Bigg) \\
=& \iota^{2}_* \Bigg( \sum_{i,j \geq 1} a_{ij} x_1^i \bullet ([n]_{F_A} \cdot x_1)^{j-1} \bullet 1_{n D_1 / pt} \Bigg) \\
=& \iota^{1}_* \Bigg( \sum_{i,j \geq 1} a_{ij} x_1^i \bullet ([n]_{F_A} \cdot x_1)^{j-1} \bullet F^n_{A,\{1\}} (x_1) \bullet 1_{D_1/pt} \Bigg).
\end{align*}
Combining the above, we see that
\begin{align*}
[D \to D] &= \iota^{1}_* \Bigg( \Big(1_{D_1/pt} + F^n_{A,\{1\}} (x_1) +  \sum_{i,j \geq 1} a_{ij} x_1^i \bullet ([n]_{F_A} \cdot x_1)^{j-1} \bullet F^n_{A,\{1\}} (x_1)\Big) \bullet 1_{D_1/pt} \Bigg) \\
&= \iota^{1}_* \Big( F^{n+1}_{A,\{1\}} (x_1) \bullet 1_{D_1/pt} \Big),
\end{align*}
where the last equality follows from the fact that $x F^m_{A,\{1\}} (x) = [m]_{F_A} \cdot x$ for all $m > 0$.

Let us then proceed by induction on $r$. Similarly to the argument above, we can find a quasi-smooth morphism $D \to \Vc^d_A$ giving us the derived Cartesian diagram
$$
\begin{tikzcd}
n_1 D_1 \arrow[]{d} \arrow[hook]{r}{\iota^1} & D \arrow[]{d} & \arrow[hook']{l}[swap]{\iota^2} D' \arrow[]{d}  \\
\Vc^d_{A,1} \arrow[hook]{r} & \Vc^d_{A} & \arrow[hook']{l} \Vc^d_{A,2}, 
\end{tikzcd}
$$ 
where $D' = n_2 D_2 + \cdots + n_r D_r$. Therefore 
\begin{align*}
[D \to D] &= [n_1 D_1 \to D] + [D' \to D] \\
&+ \iota^{12}_* \Bigg( \sum_{i,j \geq 1} a_{ij} e\big(\Oc(n_1 D_1)\big)^{i-1} \bullet e\big(\Oc(D')\big)^{j-1} \bullet 1_{n_1 D_1 \cap D'} \Bigg)
\end{align*}
in $\Omega^*\big(D \to \Spec(A) \big)$, where $\iota^{12}$ is the embedding $n_1 D_1 \cap D' \hook D$. We can then inductively compute that
\begin{align*}
[n_1 D_1 \to D] &= \iota^{\{1\}}_*\big(F^{n_1,...,n_r}_{A,\{1\}}(x_1) \bullet 1_{D_1/pt} \big) \\
[D' \to D] &= \sum_{I \subset \{2,...,r\}} \iota^I_*\big(F^{n_1,...,n_r}_{A,I}(x_2,...,x_r) \bullet 1_{D_I / pt} \big)
\end{align*}
and
\begin{align*}
&\iota^{12}_* \Bigg( \sum_{i,j \geq 1} a_{ij} e\big(\Oc(n_1 D_1)\big)^{i-1} \bullet e\big(\Oc(D')\big)^{j-1} \bullet 1_{n_1 D_1 \cap D' / pt} \Bigg) \\
=& \iota^{2}_* \Bigg( \sum_{i,j \geq 1} a_{ij} ([n_1]_{F_A} \cdot x_1)^{i} \bullet e\big(\Oc(D')\big)^{j-1} \bullet 1_{D' / pt} \Bigg) \\
=& \sum_{I \subset \{2,...,r\}} \iota^I_*\Bigg(\sum_{i,j \geq 1} a_{ij} ([n_1]_{F_A} \cdot x_1)^{i} \bullet e\big(\Oc(D')\big)^{j-1} \bullet F^{n_1,...,n_r}_{A,I}(x_2,...,x_r) \bullet 1_{D_I / pt} \Bigg).
\end{align*}
Combining the above lets us compute that
\begin{align*}
&[D \to D] \\
=& \iota^{\{1\}}_*\big(F^{n_1,...,n_r}_{A,\{1\}}(x_1) \bullet 1_{D_1/pt} \big) + \sum_{I \subset \{2,...,r\}} \iota^I_*\big(F^{n_1,...,n_r}_{A,I}(x_2,...,x_r) \bullet 1_{D_I / pt} \big) \\
+& \sum_{I \subset \{2,...,r\}} \iota^I_*\Bigg(\sum_{i,j \geq 1} a_{ij} ([n_1]_{F_A} \cdot x_1)^{i} \bullet e\big(\Oc(D')\big)^{j-1} \bullet F^{n_1,...,n_r}_{A,I}(x_2,...,x_r) \bullet 1_{D_I / pt} \Bigg) \\
=& \sum_{I \subset \{2,...,r\}} \iota^I_*\big(F^{n_1,...,n_r}_{A,I}(x_2,...,x_r) \bullet 1_{D_I / pt} \big) \\
+& \sum_{I \subset \{2,...,r\}} \iota^{\{1\} \cup I}_*\Big( F^{n_1,...,n_r}_{A,\{1\} \cup I}(x_1,...,x_r) \bullet 1_{D_{\{1\} \cup I} / pt} \Big),
\end{align*}
where the last equality follows from the fact that 
$$e\big(\Oc(D')\big) = [n_2]_{F_A} \cdot x_2 +_{F_A} \cdots +_{F_A} [n_r]_{F_A} \cdot x_r$$
and the definition of $F_{A,I}^{n_1,...,n_r}.$ Hence
$$[D \to D] = \sum_{I \subset \{1,...,r\}} \iota^I_*\big(F^{n_1,...,n_r}_{A,I}(x_1,...,x_r) \bullet 1_{D_I / pt} \big)$$
and we are done.
\end{proof}

The following lemma shows that $\Ic_A \subset \langle \Rc^\mathrm{snc}_A \rangle$.

\begin{lem}\label{IContainedInLSRelationsLem}
Let $A$ be a Noetherian ring of finite Krull dimension. Then $\Ic_A = \langle \Rc^{\Vc}_A \rangle$, where $\Rc^{\Vc}_A$ is the bivariant subset of $\PCob^*_A$ consisting of
$$
1_{\Vc^n_A / pt} - [\Vc^n_{A,1} \to \Vc^n_{A}] - [\Vc^n_{A,2} \to \Vc^n_{A}] - \iota^{12}_* \Bigg( \sum_{i,j \geq 1} a_{ij} e(\Nc_1)^{i-1} \bullet e(\Nc_2)^{j-1} \bullet 1_{\Vc^n_{A,12}/pt}\Bigg)
$$
in $\PCob^A_*(\Vc^n_A)$.
\end{lem}
\begin{proof}
Clearly $\langle \Rc^{\Vc}_A \rangle$ contains elements 
$$s_Y = 1_{\Vc^n_Y / Y} - [\Vc^n_{Y,1} \to \Vc^n_{Y}] - [\Vc^n_{Y,2} \to \Vc^n_{Y}] - \iota^{12}_* \Bigg( \sum_{i,j \geq 1} a_{ij} e(\Nc_1)^{i-1} \bullet e(\Nc_2)^{j-1} \bullet 1_{\Vc^n_{Y,12}/Y}\Bigg)$$
in $\PCob^*_A(\Vc^n_Y \to Y)$, because bivariant ideals are stable under bivariant pullbacks. Therefore, if we are given a projective morphism $f: V \to X$ so that the composition $V \to Y$ is quasi-smooth, and a quasi-smooth $Y$-morphism $V \to \Vc^n_Y$, we can compute that
\begin{align*}
f_* \big(1_{V / \Vc^n_Y} \bullet s_Y\big) &= [V \to X] - [V_1 \to X] - [V_2 \to X] \\
 &- f^{12}_* \Bigg( \sum_{i,j \geq 1} a_{ij} e(\Nc_{V_{12}/V_1})^{i-1} \bullet e(\Nc_{V_{12}/V_2})^{j-1} \bullet 1_{V_{12}/Y}\Bigg)
\end{align*}
is contained in $\langle \Rc^{\Vc}_A \rangle$, proving the claim.
\end{proof}

Combining Lemmas \ref{LSRelationsContainedInILem} and \ref{IContainedInLSRelationsLem}, we obtain the following result.

\begin{thm}\label{BivACobAltConsThm}
Let $A$ be a Noetherian ring of finite Krull dimension. Then the bivariant ideals $\Ic_A \subset \PCob^*_A$ and $\langle \Rc^\mathrm{snc}_A \rangle \subset \PCob^*_A$ coincide. \qed
\end{thm}

As a corollary, we get the desired independence of bivariant algebraic cobordism groups from the base ring.

\begin{cor}\label{BivACobIndFromACor}
Let $X \to Y$ be a morphism of derived schemes with $X$ and $Y$ quasi-projective over Noetherian rings $A$ and $B$ of finite Krull dimension. Then
$$\Omega^*_A(X \to Y) = \Omega^*_B(X \to Y).$$
\end{cor}
\begin{proof}
Clearly $\Ic_A(X \to Y) = \Ic_B(X \to Y)$, so the claim follows from Theorem \ref{BivACobAltConsThm}.
\end{proof}

\subsection{Algebraic cobordism for divisorial derived schemes}\label{BivCobConsSubsect}

In this section, we define bivariant algebraic cobordism in great generality. We will denote by $\Fs_a$ the functoriality $(\Cc_a, \Cs_a, \Is_a, \Ss_a)$, where $\Cc_a$ is the full subcategory of the $\infty$-category derived schemes consisting of finite dimensional Noetherian derived schemes having an ample family of line bundles, $\Cs_a$ consists of projective morphisms in $\Cc_a$, $\Is_a$ consists of all derived Cartesian squares in $\Cc_a$ and $\Ss_a$ consists of all morphisms in $\Cc_a$ whose relative cotangent complex has Tor-dimension 1. Consider the following universal theory defined by Yokura in \cite{yokura09} (see also \cite{annala-pre-and-cob} Section 2.2 and the beginning of Section 3).

\begin{defn}
Let us define the \emph{universal additive bivariant theory} $\Mc^\bullet_{\Fc_a, +}$ as the stably oriented bivariant theory so that $\Mc^\bullet_{\Fc_a, +}(X \to Y)$ is the group completion of the Abelian monoid on symbols
$$[V \to X]$$
on equivalence classes of projective morphisms $V \to Y$ with $\Lb_{V/Y}$ of Tor-dimension $1$, and where the monoid operation is given by disjoint union. The bivariant operations and the orientation is defined as in Definition \ref{UnivAPrecobDef}.
\end{defn}

\begin{rem}
The careful reader must have noticed that we did not define a grading on the bivariant groups $\Mc^\bullet_{\Fc_a, +}(X \to Y)$. This is because if $V \to Y$ is a general morphism of Noetherian derived schemes, then $\Lb_{V/Y}$ having Tor-dimension one does not imply that it is perfect, implying that the virtual rank of $\Lb_{V/Y}$ need not to make sense, and therefore we do not have a sensible notion of relative virtual dimension. Of course, if $V \to Y$ happens to be of finite type, then $\Lb_{V/Y}$ is perfect by \ref{QSmCharProp}, so the problem disappears. Hence, if $X \to Y$ is a finite type morphism in $\Cc_a$, we can define a grading on $\Mc^\bullet_{\Fc_a, +}(X \to Y)$
by saying that $[V \to X]$ has degree $d$ if $V \to Y$ has relative virtual dimension $-d$. We will denote the graded Abelian groups thus obtained by $\Mc_{\Fc_a, +}^*(X \to Y)$. In particular, the associated cohomology rings are graded rings. Similar remarks apply to the bivariant theories $\PCob^\bullet$ and $\Omega^\bullet$ defined below.
\end{rem}

Let us recall the following universal property from \cite{yokura09}.

\begin{thm}
$\Mc^\bullet_{\Fc_a, +}$ is the universal stably oriented additive bivariant theory with functoriality $\Fs_a$. In other words, if $\Bb^\bullet$ is a stably oriented additive bivariant theory with functoriality $\Fs_a$, then there exists a unique orientation preserving Grothendieck transformation
$$\Mc^\bullet_{\Fc_a, +} \to \Bb^\bullet.$$
\end{thm}

We can then define the universal precobordism theory.

\begin{defn}\label{UnivPrecobDef}
We define the \emph{universal precobordism theory} $\PCob^\bullet$ as a quotient of $\Mc^\bullet_{\Fc_a, +}$ by enforcing the derived double point cobordism relations: given a projective morphism $W \to \Pb^1 \times X$ so that $\Lb_{W / \Pb^1 \times Y}$ has Tor-dimension 1, and virtual Cartier divisors $D_1$ and $D_2$ so that the fibre $W_\infty$ of $W \to \Pb^1$ over $\infty$ is the sum of $D_1$ and $D_2$, then
$$[W_0 \to X] = [D_1 \to X] + [D_1 \to X] - [\Pb_{D_1 \cap D_2}(\Oc(D_1) \oplus \Oc) \to X] \in \PCob^\bullet(X \to Y),$$
where $W_0$ is the fibre of $W \to \Pb^1$ over $0$, and where $\cap$ denotes the derived intersection. It is straightforward to check that these relations respect the bivariant operations.
\end{defn}

Recall that the \emph{Euler class} $e(\Ls)$ of a line bundle $\Ls$ on $X \in \Cc_a$ is defined as $[Z_s \hook X] \in \PCob^1(X)$, where $s$ is any global section of $\Ls$ and $Z_s$ is the derived vanishing locus of $s$. Before going further, we need to make the following observation.

\begin{lem}\label{NilpChernLem}
Suppose $X \in \Cc_a$ and let $\Ls$ be a a line bundle on $X$. Then the Euler class $e(\Ls) \in \PCob^1(X)$ is nilpotent.
\end{lem}
\begin{proof}
Choose line bundles $\Ls_1, ..., \Ls_r$ and sections $s_i \in \Gamma(X; \Ls_i)$ so that the $X_{s_i}$ give an affine open cover of $X$. Using Lemma \ref{LocLem}, for $N \gg 0$ and $n \gg 0$ it is possible to find $n$ global sections $s'_{i,1}, ..., s'_{i,n}$ of $\Ls_i^{\otimes N} \otimes \Ls$ such that the intersection of the restriction of their derived vanishing loci to $X_{s_i}$ is empty, i.e., $e(\Ls_i^{\otimes N} \otimes \Ls)^{n}$ are represented by cobordism cycles supported outside $X_{s_i}$.

We can use Lemma \ref{GeomFGLLem} below to compute that
$$e(\Ls) =  e(\Ls_i^{\otimes N} \otimes \Ls) + \alpha_i \bullet e(\Ls_i^{\otimes -N})$$ 
and 
$$e(\Ls_i^{\otimes -N}) = \beta_i \bullet e(\Ls_i^{\otimes N})$$
so that
$$e(\Ls) =  e(\Ls_i^{\otimes N} \otimes \Ls) + \gamma_i \bullet e(\Ls_i^{\otimes N}) \in \PCob^*(X).$$
Hence
\begin{align*}
e(\Ls)^{rn} &= \prod_{i=1}^r \big( e(\Ls_i^{\otimes N} \otimes \Ls) + \gamma_i \bullet e(\Ls_i^{\otimes N}) \big)^n \\
&= \prod_{i=1}^r \Bigg( \sum_{j=0}^n e(\Ls_i^{\otimes N} \otimes \Ls)^j \bullet \gamma_i^{n-j} \bullet e(\Ls_i^{\otimes N})^{n-i} \Bigg),
\end{align*} 
and as the $i^{th}$ part of the product can clearly be represented by bordism cycles supported outside $X_{s_i}$, the product must vanish for the simple reason that $X_{s_i}$ cover $X$. This proves that $e(\Ls)$ is nilpotent.
\end{proof}

\begin{lem}\label{GeomFGLLem}
Let $X \in \Cc_a$ and let $\Ls_1$ and $\Ls_2$ be line bundles on $X$. Then
\begin{align*}
e(\Ls_1 \otimes \Ls_2) &= e(\Ls_1) + e(\Ls_2) - e(\Ls_1) \bullet e(\Ls_2) \bullet [\Pb_1 \to X] \\
&- e(\Ls_1) \bullet e(\Ls_2) \bullet e(\Ls_1 \otimes \Ls_2) \bullet \big( [\Pb_2 \to X] - [\Pb_3 \to X] \big)
\end{align*}
in $\PCob^1(X)$, where
\begin{align*}
\Pb_1 &= \Pb_X(\Ls_1 \oplus \Oc); \\
\Pb_2 &= \Pb_X(\Ls_1 \oplus (\Ls_1 \otimes \Ls_2) \oplus \Oc); \\
\Pb_3 &= \Pb_{\Pb_X(\Ls_1 \oplus (\Ls_1 \otimes \Ls_2))}(\Oc(1) \oplus \Oc).
\end{align*}
\end{lem}
\begin{proof}
See \cite{levine-pandharipande} Section 5.2 for the original proof for algebraic schemes in characteristic 0 or \cite{annala-yokura} Lemma 6.2 for a proof in derived setting.
\end{proof}

Now that we know that the Euler classes of line bundles are nilpotent, we may prove the following result using the arguments of \cite{annala-chern} Section 3. 

\begin{thm}\label{PCobFGLThm}
$\PCob^\bullet$ satisfies the formal group law axiom: in other words Euler classes of line bundles are nilpotent and there exists a formal group law
$$F(x,y) \in \PCob^*\big(\Spec(\Zb)\big)[[x,y]]$$
so that for all $X \in \Cc_a$ and for all line bundles $\Ls_1$ and $\Ls_2$ on $X$, the equality
$$e(\Ls_1 \otimes \Ls_2) = F\big(e(\Ls_1), e(\Ls_2)\big) \in \PCob^1(X)$$
holds. Moreover, since $\Fs_a$ is admissible and $\PCob^\bullet$ satisfies the section axiom (see \cite{annala-pre-and-cob}), the formal group law $F$ is unique. \qed
\end{thm}

Moreover, following \cite{annala-pre-and-cob} Section 3.2 and \cite{annala-chern} Section 3.2 one proves the following universal property for $\PCob$.

\begin{thm}\label{UnivPropOfPCobThm}
$\PCob^\bullet$ is the universal stably oriented additive bivariant theory with functoriality $\Fs_a$ satisfying the section and the formal group law axioms (as defined in \cite{annala-pre-and-cob}). In other words, if $\Bb^\bullet$ is another such a bivariant theory, then there exists a unique orientation preserving Grothendieck transformation $\PCob^\bullet \to \Bb^{\bullet}.$ \qed
\end{thm}

Before constructing the bivariant theory $\Omega^\bullet$ as a quotient of $\PCob^\bullet$, we consider the following stacky version of Construction \ref{VSchemeCons}.

\begin{cons}\label{VStackCons}
Consider the $\Zb$-stack $[\Ab^1/ \Gb_m] \times [\Ab^1/ \Gb_m]$, i.e., the stack classifying the data $(\Ls_1, s_1, \Ls_2, s_2)$, where $\Ls_i$ are line bundles and $s_i$ are global sections of $\Ls_i$. Then we define the stack $\Vc$ as the derived vanishing locus of $s_1 s_2$ inside $[\Ab^1/ \Gb_m] \times [\Ab^1/ \Gb_m]$. The derived vanishing loci of $s_i$ give (non-regular) closed immersions $\iota^i: \Vc_i \hook \Vc$, and the vanishing locus of $s_1$ and $s_2$ gives a (non-regular) closed embedding $\iota^{12}: \Vc_{12} \hook \Vc$. We will also denote
$$\Nc_1 := \Nc_{\Vc_{12} / \Vc_1}$$
and
$$\Nc_2 := \Nc_{\Vc_{12} / \Vc_2}.$$
Given a derived scheme $Y$, we will denote by $\Vc_Y$, $\Vc_{Y,i}$ and $\Vc_{Y,12}$ the corresponding stacks over $Y$.
\end{cons}

Armed with these stacks, we can make the following definition.

\begin{defn}\label{BivCobDef}
We define \emph{bivariant algebraic cobordism} $\Omega^\bullet$ as a quotient of $\PCob^\bullet$ by enforcing the following relations: let $X \to Y$ be a morphism in $\Cc_a$, let $f: V \to X$ be a projective morphism with $\Lb_{V/Y}$ of Tor-dimension one, and let $V \to \Vc^n_Y$ be a $Y$-morphism with $\Lb_{V / \Vc^n_Y}$ of Tor-dimension 1. We will denote by $V_i$ and $V_{12}$ the derived fibre products $\Vc_{Y,i} \times_{\Vc_Y} V$ and $\Vc_{Y,12} \times_{\Vc_Y} V$ respectively, and by $f^{12}$ the induced morphism $V_{12} \to X$. Given such data, we enforce the relation
$$[V \to X] = [V_1 \to X] + [V_2 \to X] + f^{12}_* \Bigg( \sum_{i,j \geq 1} a_{ij} e(\Nc_{V_{12}/V_1})^{i-1} \bullet e(\Nc_{V_{12}/V_2})^{j-1} \bullet 1_{V_{12}/Y}\Bigg)$$
in $\Omega^\bullet(X \to Y)$, where $a_{ij}$ are the coefficients of the formal group law $F$. It is straightforward to check that these relations respect the bivariant operations.
\end{defn}

Before continuing, we have to make sure that $\Omega^\bullet$ as defined above really extends $\Omega^*_A$ defined earlier.

\begin{lem}
Let $A$ be a Noetherian ring of finite Krull dimension, and let $X \to Y$ be a morphism of quasi-projective derived $A$-schemes. Then
$$\Omega^*(X \to Y) = \Omega^*_A(X \to Y).$$
\end{lem}
\begin{proof}
The issue is that $\Omega^*(X \to Y)$ was defined using the stack $\Vc_Y$ whereas $\Omega^*_A(X \to Y)$ was defined using the schemes $\Vc^n_Y$. However, since any line bundle on a quasi-projective derived scheme is of form $\Ls_1 \otimes \Ls_2^\vee$ with $\Ls_i$ globally generated, it follows that any morphism $X \to \Vc_Y$ factors through the canonical morphism $\Vc^n_Y \to \Vc_Y$ for $n \gg 0$, and therefore $\Omega^*(X \to Y)$ and $\Omega^*_A(X \to Y)$ have the same relations.
\end{proof}

As in \cite{annala-pre-and-cob}, the following result is an immediate consequence of Theorem \ref{UnivPropOfPCobThm}.

\begin{cor}\label{UnivPropOfCobCor}
$\Omega^\bullet$ is the universal stably oriented additive bivariant theory with functoriality $\Fs_a$ satisfying the section, the formal group law and the snc axioms (as defined in \cite{annala-pre-and-cob}). In other words, if $\Bb^\bullet$ is another such a bivariant theory, then there exists a unique orientation preserving Grothendieck transformation $\PCob^\bullet \to \Bb^\bullet.$ \qed
\end{cor}

Bivariant algebraic cobordism in this generality allows us to define the corresponding absolute homology theory.

\begin{defn}\label{AlgBordDef}
Let $X \in \Cc_a$. We define the \emph{algebraic bordism} of $X$ as the Abelian group
$$\Omega_\bullet(X) := \Omega^\bullet \big(X \to \Spec(\Zb)\big).$$
We do not define a grading on $\Omega_\bullet(X)$.
\end{defn}

\begin{rem}
By Proposition \ref{DCompIntAltCharProp}, we see that algebraic bordism $\Omega_\bullet(X)$ classifies derived complete intersection schemes projective over $X$ up to ``cobordism''. A natural grading would be given by the virtual dimension of the cycles, but unfortunately this is not well defined in general, as noted in Remark \ref{DimProblemRem}. However, at least when $X$ is of finite type over a field $k$, the grading by virtual dimension makes sense. We will denote the graded Abelian groups obtained this way by $\Omega_*(X)$. If $X$ is quasi-projective over $k$, then these graded groups agree with the ones defined by Lowrey and Schürg in \cite{lowrey-schurg} by the results of \cite{annala-pre-and-cob}. In the special case where $k$ has characteristic 0, $\Omega_*(X)$ is canonically isomorphic to the algebraic bordism group $\Omega_*^\mathrm{LM}(X_\cl)$ of Levine--Morel.
\end{rem}

\subsection{Nilpotency of Euler classes on divisorial schemes and applications}\label{NilpEulerClassSubSect}

The purpose of this section is to analyze the potential failure of Euler classes of vector bundles to be nilpotent in $\PCob^\bullet$, and to construct theories $\PCob'^\bullet$ and $\Omega'^\bullet$ where the desired nilpotence is provably true. Note that nilpotence of Euler classes is of crucial importance in Section \ref{GeneralPBFSect}, and most of the proofs fail without it.

 Let us start by recalling that given $X \in \Cc_a$ and a vector bundle $E$ of rank $r$ on $X$, then the \emph{Euler class} of $E$ is defined as
$$e(E) := [Z_s \hook X] \in \PCob^r(X),$$
where $s$ is an arbitrary global section of $E$, and $Z_s$ is the derived vanishing locus of $s$. We also recall that these classes are multiplicative in exact sequences of vector bundles by \cite{annala-chern} Lemma 4.1. Note that by Lemma \ref{NilpChernLem}, Euler classes of line bundles are nilpotent, but the proof does not seem to generalize to higher rank bundles. Moreover, the proof of the nilpotence of Euler classes given in \cite{annala-chern} Lemma 4.2 does not seem to generalize either, unless we restrict our attention to derived schemes admitting an ample line bundle. Hence we we need to come up with another argument.

 %the behavior of powers of Euler classes of vector bundles on derived schemes admitting an ample family of line bundles. By Lemma \ref{NilpChernLem}, Euler classes of line bundles are nilpotent. We will investigate what happens when we try to prove the nilpotency for higher rank vector bundles.

Suppose $E$ is a rank $r > 1$ vector bundle. Then, as there exists an exact sequence
$$0 \to \Oc(-1) \to E \to Q \to 0$$
of vector bundles on $\Pb(E)$, the nilpotency of 
$$e(E) \in \PCob^*(\Pb(E))$$
follows from the nilpotency of Euler classes of line bundles. We would be done if we could show that the pullback $\PCob^*(X) \to \PCob^*(\Pb(E))$ is injective, which we will try to prove by constructing a class $\eta \in \PCob^*(\Pb(E))$ pushing forward to $1_X$. We fail to achieve this, but instead we find a family of bivariant classes in $\PCob^\bullet$ which have to be killed.

We will set up the following notation: $Z$ is the derived vanishing locus of a global section of $E$, $\wtil X$ is the derived bow up $\bl_Z(X)$, and  $\Ec \hook \wtil X$ is the exceptional divisor. We begin by recalling the following formula, which is easy to prove using the double point cobordism formula.

\begin{lem}\label{BlowUpFormulaLem}
Let everything be as above. Then
$$1_X = [\wtil X \to X] + e(E) \bullet [\Pb(E \oplus \Oc) \to X] - [\Pb_{\Ec}(\Oc(\Ec) \oplus \Oc) \to X]$$
in $\PCob^*(X)$. \qed
\end{lem}

Since we do not know that $e(E) \in \PCob^*(X)$ is nilpotent, the argument of \cite{annala-chern} Lemma 4.5 gives us only the following truncated version.

\begin{lem}\label{TruncatedBlowUpLiftLem}
Let everything be as above. Then, for all $n \geq 0$, 
\begin{align*}
1_X &= \sum_{i=0}^n e(E)^i \bullet [\Pb(E \oplus \Oc) \to X]^i \bullet \big( [\wtil X \to X] - [\Pb_\Ec(\Oc(\Ec) \oplus \Oc) \to X] \big) \\
&+ e(E)^{n+1} \bullet [\Pb(E \oplus \Oc) \to X]^{n+1}
\end{align*}
in $\PCob^*(X)$. \qed
\end{lem}

Let us define
$$\eta_n := \sum_{i=0}^n e(E)^i \bullet [\Pb_{\Pb(E)}(E \oplus \Oc) \to \Pb(E)]^i \bullet \big( [\wtil X \to \Pb(E)] - [\Pb_\Ec(\Oc(\Ec) \oplus \Oc) \to \Pb(E)] \big)$$
in $\PCob^*(\Pb(E))$. Since $e(E) \in \PCob^*(\Pb(E))$ is nilpotent, $\eta_n$ stabilize for $n \gg 0$, and as
$$1_X - \pi_!(\eta_n) = e(E)^{n+1} \bullet [\Pb(E \oplus \Oc) \to X]^{n+1} \in \PCob^*(X),$$
we conclude that also the right hand side stabilizes when $n$ tends to infinity. We can therefore make the following definition.

\begin{defn}
We define
$$\epsilon_{E, X} \in \PCob^0(X)$$
as the limit of $e(E)^n \bullet [\Pb(E \oplus \Oc) \to X]^n$ when $n \to \infty$. Define $n_{X,E} \geq 0$ as the smallest integer $n$ such that $\epsilon_{E, X} = e(E)^n \bullet [\Pb(E \oplus \Oc) \to X]^n$. 
\end{defn}

\begin{defn}
We define the bivariant theory $\PCob'^\bullet$ as the quotient of $\PCob^\bullet$ by the bivariant ideal $\Ic_\epsilon$ generated by all the elements $\epsilon_{E,X}$, where $X \in \Cc_a$ and $E$ is a vector bundle on $X$.
\end{defn}

We record the following expression of $\Ic_\epsilon(X \to Y)$, which might be useful.

\begin{prop}
Let $X \to Y$ be a map in $\Cc_a$. Then $\Ic_\epsilon(X \to Y)$ is the subgroup of $\PCob^\bullet(X \to Y)$ generated by
$$f_* \big( e(E)^n \bullet [\Pb(E \oplus \Oc) \to V]^n \bullet 1_{V/Y}\big)$$
where $f: V \to X$ is any projective morphism so that $\Lb_{V/Y}$ has Tor-dimension 1, $E$ is a vector bundle on $V$ and $n \geq n_{E,V}$.
\end{prop}
\begin{proof}
It is enough to show that these groups form a bivariant ideal, i.e., they are stable under pushforwards, pullbacks and left and right multiplication by $\PCob^\bullet$. Stability under pullbacks and pushforwards is shown similarly to \cite{annala-cob} Lemma 3.8 using the fact that for $W \to V$, the inequality
$$n_{E, W} \leq n_{E, V}$$
holds.

To prove stability under left multiplication, it suffices to consider multiplication by 
$$[W \stackrel h \to Z] \in \Omega^\bullet(Z \to X).$$
To investigate this, we form the derived Cartesian diagram
$$
\begin{tikzcd}
V \arrow[]{r}{f} & X \arrow[]{r} & Y \\
Z' \arrow[]{r}{f'} \arrow[]{u}[swap]{g'} & Z \arrow[]{u}[swap]{g} \\
W' \arrow[]{r}{f''} \arrow[]{u}[swap]{h'} & W \arrow[]{u}[swap]{h} 
\end{tikzcd}
$$
and since a direct computation shows that
\begin{align*}
&[W \to Z] \bullet f_* \big( e(E)^n \bullet [\Pb_V(E \oplus \Oc) \to V]^n \bullet 1_{V/Y}\big) \\
=& (h \circ f'')_* \big( e(E)^n \bullet [\Pb_{W'}(E \oplus \Oc) \to W']^n \bullet 1_{W'/Y}\big),
\end{align*}
stability under left multiplication follows from the inequality $n_{E,W'} \leq n_{E,V}$. The proof that these groups are stable under right multiplication is similar.
\end{proof}

Note that by construction $\PCob'^\bullet$ satisfies the following.

\begin{lem}\label{BlowUpLiftLem}
Suppose $X$ is a finite dimensional divisorial Noetherian derived scheme, and let $E$ be a vector bundle on $X$. Then the class
$${1_{\Pb(E)} - e(\Oc(-1)) \bullet [\Pb_{\Pb(E)}(\Oc(-1) \oplus \Oc) \to \Pb(E)] \over 1_{\Pb(E)} - e(E) \bullet [\Pb_{\Pb(E)}(E \oplus \Oc) \to \Pb(E)]} \bullet e(Q) \in \PCob'^\bullet\big(\Pb(E)\big)$$
pushes forward to $1_X$.
\end{lem}
\begin{proof}
Indeed, this follows from Lemma \ref{TruncatedBlowUpLiftLem}, as derived vanishing loci of $Q$ correspond to derived blow ups of $X$.
\end{proof}

We constructed $\PCob'^\bullet$ so that the following result would be true.

\begin{prop}\label{NilpEulerProp}
Let $X \in \Cc_a$, and let $E$ be a vector bundle on $X$. Then
$$e(E) \in \PCob'^*(X)$$
is nilpotent.
\end{prop}
\begin{proof}
We already know that $e(E) \in \PCob'^*(\Pb(E))$ is nilpotent. Moreover, it follows from Lemma \ref{BlowUpLiftLem} and the projection formula that the pullback morphism $\PCob'^*(X) \to \PCob'^*(\Pb(E))$ is an injection, and therefore $e(E)$ is nilpotent in $\PCob'^*(E)$ as well.
\end{proof}

We also record the following useful observation, which also leads to universal properties for $\PCob'^\bullet$ and $\Omega'^\bullet$.

\begin{lem}\label{EquivalentModificationsLem}
Let $\nu: \PCob^\bullet \to \Bb^\bullet$ be a Grothendieck transformation. Then the following conditions are equivalent:
\begin{enumerate}
\item for all $X \in \Cc_a$ and all vector bundles $E$ over $X$, $\nu(e(E)) \in \Bb^*(X)$ is nilpotent; 
\item for all $X \in \Cc_a$ and all vector bundles $E$ over $X$, $\nu(\epsilon_{E,X}) = 0 \in \Bb^*(X)$; 
\item for all $X \in \Cc_a$ and all vector bundles $E$ over $X$, there exists $\eta \in \Bb^*(\Pb(E))$ pushing forward to $1_X \in \Bb^*(X)$;
\item for all $X \in \Cc_a$ and all vector bundles $E$ over $X$, the pullback morphism $\Bb^*(X) \to \Bb^*(\Pb(E))$ is injective.
\end{enumerate}
\end{lem}
\begin{proof}
The implications $1 \Rightarrow 2 \Rightarrow 3$ follow directly from the definition of $\epsilon_{E,X}$. The implication $3 \Rightarrow 4$ follows easily from the projection formula, and the implication $4 \Rightarrow 1$ follows from the fact that
$$\nu \big(e(E) \big) = \nu\big(e(\Oc(-1)) \bullet e(Q) \big) \in \Bb^*(\Pb(E)),$$
which is clearly nilpotent.
\end{proof}

We can then prove the following universal properties (compare this to Theorem \ref{UnivPropOfPCobThm} and Corollary \ref{UnivPropOfCobCor}).

\begin{thm}\label{UnivPropOfModiefiedCobThm}
The bivariant theory $\PCob'^\bullet$ ($\Omega'^\bullet$) is the universal stably oriented additive bivariant theory with functoriality $\Fs_a$ satisfying the section and the formal group law (and the snc) axioms where the pullback morphism along $\Pb(E) \to X$ is injective for all $X \in \Cc_a$ and all vector bundles $E$ on $X$. \qed
\end{thm}

In particular both $\PCob'^\bullet$ and $\Omega'^\bullet$ admit a unique orientation preserving Grothendieck transformation to the bivariant algebraic K-theory of derived schemes restricted to $\Fs_a$. We end with the following natural question.

\begin{quest}
Does the bivariant theory $\PCob^\bullet$ coincide with $\PCob'^\bullet$? In other words, does $\PCob^\bullet$ satisfy any of the equivalent conditions of Lemma \ref{EquivalentModificationsLem}?
\end{quest}

\section{Projective bundle formula and applications}\label{GeneralPBFSect}

In this section we prove the projective bundle formula (Theorem \ref{GeneralPBFThm}) for $\PCob'^\bullet$ (in fact for all bivariant theories that are quotients of $\PCob'^\bullet$) and use this to define Chern classes of vector bundles in $\PCob'^*$ satisfying good properties by Theorem \ref{GeneralChernClassThm}. As a standard corollary, we obtain a generalized version of the cohomological Conner--Floyd theorem (Theorem \ref{GeneralCFThm}) and the cohomological Grothendieck--Riemann--Roch theorem (Theorem \ref{GeneralGRRThm}).

The proofs use similar ideas as the proofs of the corresponding claims in \cite{annala-chern}. The two main differences are as follows
\begin{enumerate}
\item since we can no longer assume the existence of an ample line bundle, many of the old proofs had to be modified;
\item we have restructured the arguments from \cite{annala-chern} so that we first prove the projective bundle formula, and only after that construct Chern classes and prove their properties.
\end{enumerate}
The structure of this section is as follows: in Section \ref{FundEmbSubSect} we construct the fundamental embedding, which realizes the bivariant group of $\Pb(E) \to Y$ as a subgroup of the bivariant group of $\Pb^\infty \times X \to Y$. In Section \ref{GenPBFSubSect} we prove the projective bundle formula by first proving it for trivial projective bundles, and then use the fundamental embedding to prove the general case. Finally, in Section \ref{GenChernClassSubSect} we use the projective buldne formula to define and study Chern classes of vector bundles, and prove the standard corollaries (Conner--Floyd, Grothendieck--Riemann--Roch).

\subsection{The fundamental embedding}\label{FundEmbSubSect}

The purpose of this section is to construct an injective homomorphism 
$$\iota_E: \Bb^\bullet (\Pb(E) \to Y) \to \Bb^\bullet (\Pb^\infty \times X \to Y) := \colim_{n \geq 0} \Bb^\bullet (\Pb^n \times X \to Y),$$
called the \emph{fundamental embedding}, where $\Bb^\bullet$ is a quotient of $\PCob^\bullet$, $X \to Y$ is a morphism in $\Cc_a$, $E$ is a vector bundle on $X$, and the colimit on the right hand side is given by bivariant pushforwards along the evident sequence of inclusions
$$\Bb^\bullet(X \to Y) \hook \Bb^\bullet(\Pb^1 \times X \to Y) \hook \Bb^\bullet(\Pb^2 \times X \to Y) \to \cdots$$
Note that the structure morphisms are injections by \cite{annala-chern} Proposition 2.14.

To do so, we consider the diagram
$$
\begin{tikzcd}
&  \Bb^\bullet\big( \Pb^{\infty} \times X \to Y\big) \arrow[hook]{d}{j_E} \\
\Bb^\bullet\big(\Pb(E) \to Y\big) \arrow[hook]{r}{i_E} & \colim_{n \geq 0} \Bb^\bullet\big( \Pb(\Oc^{\oplus n} \oplus E) \to Y\big)
\end{tikzcd}
$$
where all the morphisms (the structure morphisms included) are given by pushing forward along the obvious inclusions, and injectivity of the  transformations follows from \cite{annala-chern} Proposition 2.14. The main step of the construction is the following result.

\begin{lem}\label{JSurjLem}
Let everything be as above. Then $j_E$ is a bijection.
\end{lem}

Assuming this result for now, we can make the following definition.

\begin{defn}
Let everything be as above. Then we define the \emph{fundamental embedding} as $\iota_E := j_E^{-1} \circ i_E$.
\end{defn}

This construction has many good properties, of which we are going to record only a fraction below. Recall that for any vector bundle $E$ on $X$, the bivariant group $\Bb^\bullet(\Pb(E)  \to Y)$ has a natural structure of a $\Bb^*(X)$-module by
$$\alpha . \beta := \pi_E^*(\alpha) \bullet \beta,$$
where $\pi_E$ is the natural morphism $\Pb(E) \to X$.

\begin{prop}\label{FundamentalEmbProp}
Let everything be as above. Then
\begin{enumerate}
\item $\iota_E$ is a map of $\Bb^*(X)$-modules;
\item there is a well defined operator $e(\Oc(1)) \bullet -$ on $\Bb^\bullet(\Pb^\infty \times X \to Y)$, and moreover 
$$e(\Oc(1)) \bullet \iota_E(-) = \iota_E\big(e(\Oc(1)) \bullet -\big);$$
\item $\iota_E$ is compatible with bivariant multiplication from right in the sense that if $\alpha \in \Bb^\bullet(Y \to Z)$, then the square
$$
\begin{tikzcd}
\Bb^\bullet(\Pb(E) \to Y) \arrow[]{d}{- \bullet \alpha} \arrow[]{r}{\iota_E} & \Bb^\bullet(\Pb^\infty \times X \to Y) \arrow[]{d}{- \bullet \alpha}  \\
\Bb^\bullet(\Pb(E) \to Z) \arrow[hook]{r}{\iota_E} & \Bb^\bullet(\Pb^\infty \times X \to Z) 
\end{tikzcd}
$$
commutes;
\item $\iota_E$ is compatible with bivariant pullbacks;
\item if $\psi: E \hook F$ is an inclusion of vector bundles, then
$$\iota_E = \iota_F \circ \Pb(\psi).$$
\end{enumerate}
\end{prop}
\begin{proof}
The first and the second claim follow easily from bivariant projection formula $A_{123}$. The third and the fourth claim follow from $A_{12}$ and $A_{23}$ respectively. To prove the fifth claim, we consider the diagram
$$
\begin{tikzcd}
& \Bb^\bullet\big( \Pb^{\infty} \times X \to Y\big) \arrow[hook]{d}{j_E} \\
\Bb^\bullet\big(\Pb(E) \to Y\big) \arrow[hook]{r}{i_E} \arrow[hook]{d}{\Pb(\psi)_*} & \colim_{n \geq 0} \Bb^\bullet\big( \Pb(\Oc^{\oplus n} \oplus E) \to Y\big) \arrow[hook]{d}{\Pb(\psi')_*}\\
\Bb^\bullet\big(\Pb(F) \to Y\big) \arrow[hook]{r}{i_F} & \colim_{n \geq 0} \Bb^\bullet\big( \Pb(\Oc^{\oplus n} \oplus F) \to Y\big),
\end{tikzcd}
$$
where the square commutes by construction. Since $\Pb(\psi')_* \circ j_E$ is an isomorphism by Lemma \ref{JSurjLem}, it follows that also $\Pb(\psi')_*$ is an isomorphism, and the claim follows from the fact that $j_F = \Pb(\psi')_* \circ j_E$.
\end{proof}

Our next goal is to prove Lemma \ref{JSurjLem}. However, we need some preliminary results and definitions before giving the proof. The following observation shows how the hyperplane bundle $\Oc(1)$ on $\Pb(E)$ controls the bivariant cobordism classes.

\begin{lem}\label{PBClassControlledByLLem}
Let $X \to Y$ be a morphism in $\Cc_a$ and let $E$ a vector bundle on $X$. Let $V \to X$ be a projective morphism with $\Lb_{V/Y}$ of Tor-dimension 1, and let $s_1,s_2: E^\vee \to \Ls$ be two surjections of vector bundles on $V$ ($\Ls$ is a line bundle on $V$) inducing projective morphisms $f_1, f_2: V \to \Pb(E)$ over $X$. Then
$$[V \stackrel{f_1}{\to} \Pb(E)] = [V \stackrel{f_2}{\to} \Pb(E)]$$
in $\PCob^\bullet(\Pb(E) \to Y)$.
\end{lem}
\begin{proof}
Let us define a morphism
$$s := x_0 s_1 + x_1 s_2: E^\vee \to \Ls(1)$$
of vector bundles on $\Pb^1 \times V$. As the derived vanishing locus $Z$ of $s^\vee$ is disjoint from the fibres over $0$ and $\infty$, we can use Proposition \ref{BlowUpIsResolutionSchemeProp} to find a surjection 
$$E^\vee \to \Ls(1 - \Ec)$$
of vector bundles on $\bl_Z(\Pb^1 \times V)$ inducing  a projective morphism $\bl_Z(\Pb^1 \times V) \to \Pb^1 \times \Pb(E)$ which realizes the desired relation.
\end{proof}

The following notation will be useful in the proof.

\begin{defn}
Given a derived scheme $X$ and a line bundle $\Ls$ on $X$, we recursively define the following pairs:
$$\big(P_0(X,\Ls), \Ms_0(X,\Ls) \big) := \big(X, \Ls \big)$$
and 
$$\big(P_{i+1}(X,\Ls), \Ms_{i+1}(X,\Ls) \big) := \big(\Pb_{P_i(X,\Ls)} (\Ms_i(X,\Ls) \oplus \Oc), \Ms_i(X,\Ls)(H_{i+1}) \big),$$
where $\Oc(H_{i+1})$ is the hyperplane bundle on $\Pb_{P_i(X,\Ls)} \big(\Ms_i(X,\Ls) \oplus \Oc\big)$.
\end{defn}

\begin{proof}[Proof of Lemma \ref{JSurjLem}]
To simplify notation, we set
$$\Bb^\bullet( \Pb(\Oc^{\oplus \infty} \oplus E) \to Y) := \colim_{n \geq 0} \Bb^\bullet( \Pb(\Oc^{\oplus n} \oplus E) \to Y).$$
Moreover, by Lemma \ref{PBClassControlledByLLem}, given a projective morphism $V \to X$ with $\Lb_{V/Y}$ of Tor-dimension 1 and a line bundle $\Ls$ on $V$ that admits a surjection from $E^\vee \oplus \Oc^{\oplus n}$ for some $n$, there exists a unique class
$$[V \to X; \Ls] := [X \stackrel f \to \Pb(\Oc^{\oplus \infty} \oplus E)] \in \Bb^\bullet( \Pb(\Oc^{\oplus \infty} \oplus E) \to Y),$$
where $f^* (\Oc(1)) \simeq \Ls$. As $\Bb^\bullet$ was assumed to be a quotient of $\PCob^\bullet$, these classes generate, so it is enough to show that all $[V \to X; \Ls]$ lie in the image of $j_E$. Note that this is clear if $\Ls$ is globally generated.

Our argument is a variation of the arguments of Section 6.2 of \cite{annala-yokura}. Let $D$ be the derived vanishing locus of a global section of $\Ls$, $W$ the derived blow up of $\Pb^1 \times V$ at $\infty \times D$, and consider the line bundle $\Ls' := \Ls(H - \Ec)$ on $W$, where $H$ is the pullback of the hyperplane class on $\Pb^1$. Since the center of the blow up giving $W$ is the derived vanishing locus of a global section of $\Ls \oplus \Oc(H)$, there exists by Proposition \ref{BlowUpIsResolutionSchemeProp} a surjection 
$$\Oc(H) \oplus \Ls \twoheadrightarrow \Ms_0$$
of vector bundles on $W$. Hence, there exists a surjection of vector bundles
$$\Oc^{\oplus n + 2} \oplus E^\vee \twoheadrightarrow \Ms_0.$$
Moreover, since there exists a natural surjection
$$\Oc \oplus \Ms_i(W, \Ls') \twoheadrightarrow \Ms_{i+1}(W, \Ls')$$
on $P_{i+1}(W, \Ls')$, we can recursively find surjections
$$\Oc^{\oplus n + 2 + i + 1} \oplus E^\vee \twoheadrightarrow \Ms_{i+1}(W, \Ls'),$$
which give rise to projective morphisms $P_i(W,\Ls') \to \Pb^1 \times \Pb(\Oc^{\oplus n + 2 + i} \oplus E)$ for all $i \geq 0$.

Since the line bundle $\Ls'$ on $W$ restricts to
\begin{enumerate}
\item $\Ls$ on $0 \times V$;
\item $\Oc$ on the strict transform of $\infty \times V$;
\item $\Ms_1(V,\Ls)\vert_\Ec$ on the exceptional divisor $\Ec \simeq P_1(V, \Ls) \vert_D$,
\end{enumerate} 
we see that the morphisms $P_i(W,\Ls') \to \Pb^1 \times \Pb(\Oc^{\oplus n + 2 + i} \oplus E)$ give rise to the relations
\begin{align*}
[P_i(V, \Ls) \to X; \Ms_i(V, \Ls)] &= [P_i(V, \Oc) \to X; \Ms_i(V, \Oc)] \\
&+ e(\Ls) \bullet [P_{i+1}(V, \Ls) \to X; \Ms_{i+1}(V, \Ls)] \\
&- e(\Ls) \bullet [\Pb_{P_i(V, \Oc)}(\Ls \oplus \Oc) \to X; \Ms_i(V, \Oc)]
\end{align*}
in $\Bb^\bullet\big( \Pb(\Oc^{\oplus \infty} \oplus E) \to Y\big)$, and since $e(\Ls)$ is nilpotent, we compute that
\begin{align*}
[V \to X; \Ls] &= [V \to X; \Oc] + \sum_{i=1}^\infty e(\Ls)^i \bullet \Big([P_i(V, \Oc) \to X; \Ms_i(V, \Oc)] \\
&-[\Pb_{P_{i-1}(V, \Oc)}(\Ls \oplus \Oc) \to X; \Ms_{i-1}(V, \Oc)] \Big) \in \Bb^\bullet\big( \Pb(\Oc^{\oplus \infty} \oplus E) \to Y\big).
\end{align*}
As $\Ms_i(V, \Oc)$ are globally generated, the right hand side of the above equation is evidently in the image of $j_E$, and therefore so is $[V \to X; \Ls]$. This proves the claim.
\end{proof}

\subsection{Proof of the projective bundle formula}\label{GenPBFSubSect}

The goal of this section is to prove the following result.

\begin{thm}\label{GeneralPBFThm}
Let $\Bb^\bullet$ be a quotient of $\PCob'^\bullet$, let $X \to Y$ be a morphism in $\Cc_a$ and let $E$ be a vector bundle of rank $r$ on $X$. Then the morphism
$$\Pc roj: \bigoplus_{i=0}^{r-1} \Bb^\bullet(X \to Y) \to \Bb^\bullet(\Pb(E) \to Y)$$
given by
$$(\alpha_0, \alpha_1, ..., \alpha_r) \mapsto \sum_{i=0}^{r-1} e(\Oc(1))^i \bullet 1_{\Pb(E)/X} \bullet \alpha_i$$
is an isomorphism.
\end{thm}

Throughout the section, $\Bb^\bullet$ is a quotient of $\PCob'^\bullet$, $X \to Y$ is a morphism in $\Cc_a$, and $E$ is a rank $r$ vector bundle on $X$, unless otherwise mentioned. We start by proving the result for trivial vector bundles, which follows from the next result.

\begin{lem}\label{GeneralTrivPBFLem}
Let $X \to Y$ be a morphism in $\Cc_a$, and let $\Bb^\bullet$ be a quotient of $\PCob^\bullet$. Then the morphism
$$\bigoplus_{i=0}^{\infty} \Bb^\bullet(X \to Y) \to \Bb^\bullet(\Pb^\infty \times X \to Y),$$
where the $i^{th}$ morphism is given by 
$$[\Pb^i \times X \hook \Pb^\infty \times X] \bullet -:  \Bb^\bullet(X \to Y) \to \Bb^\bullet(\Pb^\infty \times X \to Y),$$
is an isomorphism.
\end{lem}
\begin{proof}
The proof from Sections 6.2 and 6.3 of \cite{annala-yokura} goes through without essential modifications. We outline the argument here as the proof of Lemma \ref{JSurjLem} got us almost there. Indeed, by that proof, $\Bb^\bullet(\Pb^{\infty} \times X \to Y)$ is generated as an Abelian group by elements of form
$$[\Pb_i(X, \Oc) \stackrel{f_i}{\to} \Pb^\infty \times X] \bullet \alpha,$$
where $f_i^*(\Oc(1)) \simeq \Ms_i(X,\Oc)$ and $\alpha \in \Bb^\bullet(X \to Y)$. It then follows from a straightforward computation (see \cite{annala-yokura} Lemmas 6.20 and 6.21) that every element of $\Bb^\bullet(\Pb^{\infty} \times X \to Y)$ is an integral combination of 
$$[\Pb^i \times X \hook \Pb^\infty \times X] \bullet \alpha,$$
and using the operator $e(\Oc(1)) \bullet -$ and the ``bivariant pushforward'' 
$$\Bb^\bullet(\Pb^\infty \times X \to Y) \to \Bb^\bullet(X \to Y)$$
one checks that such an expression is unique, proving the claim.
\end{proof}

Combining this with the fundamental embedding constructed in the previous section, we see that
$$\iota_E(\alpha) = \sum_{i = 0}^\infty [\Pb^i \times X \hook \Pb^\infty \times X] \bullet u_i(\alpha) \in \Bb^\bullet(\Pb^\infty \times X \to Y)$$
for some uniquely determined $u_i(\alpha) \in \Bb^\bullet(X \to Y)$.

\begin{defn}\label{CoeffDef}
Given $\alpha \in \Bb^\bullet(\Pb(E) \to Y)$, we will call $u_i(\alpha)$ the \emph{$i^{th}$ coefficient} of $\alpha$.
\end{defn}

Our first goal is to show that the first $r$ coefficients completely determine an element of $\Bb^\bullet(\Pb(E) \to Y)$. To do so, we start with the following observation.

\begin{lem}\label{ChernClassFormulaLem}
Let $X \in \Cc_a$, and let $E$ be a vector bundle of rank $r$ on $X$. Then there exists $d_i(E) \in \Bb^*(X)$ so that
$$e(\Oc(1))^r = \sum_{i=1}^{r} d_i(E) \bullet e(\Oc(1))^{r-i} \in \Bb^*(\Pb(E)).$$
\end{lem}
\begin{proof}
Let $\pi$ be the natural projection $\Fl(E) \to \Pb(E)$ and let $\eta \in \Bb^*(\Fl(E))$ be such that $\pi_!(\eta) = 1_{\Pb(E)}$.
Let 
$$0 = E_0 \subset E_1 \subset \cdots \subset E_{r-1} \subset E_r = E$$
be a filtration of $E$ on $\Fl(E)$ with $\Ls_i := E_{i}/E_{i-1}$ line bundles. Then, since
\begin{align*}
e(\Ls_1(1)) \bullet \cdots \bullet e(\Ls_r(1)) &= e(E(1)) \\
&= 0
\end{align*} 
it follows that also
\begin{align*}
e(\Ls_1^\vee(-1)) \bullet \cdots \bullet e(\Ls^\vee_r(-1)) &= \inv_F \big(e(\Ls_1(1))\big) \bullet \cdots \bullet \inv_F \big( e(\Ls_r(1))\big) \\
&= 0
\end{align*}
and therefore
$$\prod_{i=1}^r \big( F(e(\Ls_i^\vee(-1)), e(\Oc(1))) - e(\Oc(1)) \big) = \prod_{i=1}^r \big( e(\Ls_i^\vee) - e(\Oc(1)) \big)$$
vanishes. In conclusion,
$$\sum_{i=0}^r (-1)^{r-i} s_i\big(e(\Ls_1^\vee),...,e(\Ls_r^\vee)\big) \bullet e(\Oc(1))^{r-i} = 0 \in \PCob^*(\Fl(E)),$$
where $s_i$ are the elementary symmetric polynomials, and the claim follows with 
$$d_i(E) := (-1)^{r-i+1} \pi_!\Big( s_i\big(e(\Ls_1^\vee),...,e(\Ls_r^\vee)\big) \bullet \eta \Big)$$
from the projection formula.
\end{proof}

It is then easy to prove that the first $r$ coefficients of $\alpha \in \Bb^\bullet(\Pb(E) \to Y)$ completely determine $\alpha$.

\begin{lem}\label{RCoeffLem}
Suppose $\alpha \in \Bb^\bullet(\Pb(E) \to Y)$ is such that $u_i(\alpha) = 0$ for all $i < r$. Then $\alpha = 0$.
\end{lem}
\begin{proof}
Clearly
$$e(\Oc(1))^r \bullet \iota_E(\alpha) = \sum_{i=0}^\infty [\Pb^i \times X \hook \Pb^\infty \times X] \bullet u_{r+i}(\alpha).$$
On the other hand, by Lemma \ref{ChernClassFormulaLem}, we can also compute that
\begin{align*}
e(\Oc(1))^r \bullet \iota_E(\alpha) &= \sum_{i=1}^r d_i(E) \bullet e(\Oc(1))^{r-i} \bullet \iota_E(\alpha)
\end{align*}
so that
$$u_{r+i}(\alpha) = \sum_{j=0}^{r-1} d_{r-j}(E) \bullet u_{j+i}(\alpha),$$
and therefore $u_i(\alpha) = 0$ for all $i$. Since the fundamental embedding is an injection, the claim follows.
\end{proof}

In order to finish the proof of the projective bundle formula, we need to show that given $u_0,...,u_r \in \Bb^\bullet(X \to Y)$, there exists 
$$\alpha := \sum_{i=0}^{r-1} e(\Oc(1))^i \bullet 1_{\Pb(E)/X} \bullet \alpha_i \in \Bb^\bullet(\Pb(E) \to Y)$$
so that $u_i(\alpha) = u_i$ for all $i < r$. To achieve this, we start by investigating the image of the fundamental class under $\iota_E$.

\begin{lem}\label{CoefficientLem1}
Suppose $X \in \Cc_a$ and let $E$ be a vector bundle of rank $r$ on $X$. Then $u_{r-1}(1_{\Pb(E)/X})$ is a unit and $u_{i}(1_{\Pb(E)/X})$ is nilpotent for all other $i$.
\end{lem}
\begin{proof}
We start with the case of a split vector bundle, i.e.,
$$E \simeq \bigoplus_{i=1}^r \Ls_i.$$
By construction of $\iota_E$, it is enough to understand the image of $1_{\Pb(E)/X}$ inside
$$\Bb^*(\Pb(\Oc^\infty \oplus E) \to X) :=  \colim_{n \geq 0} \Bb^*(\Pb(\Oc^n \oplus E) \to X).$$
If $E' = \bigoplus_{i \in I} \Ls_i$ for some $I \subset \{1,...,r\}$ and $n \geq 0$, we will denote by 
$$[\Pb(\Oc^{\oplus n} \oplus E') \hook \Pb(\Oc^\infty \oplus E)]$$
the class of the obvious linear embedding. 

Let $E' := \bigoplus_{i=2}^r \Ls_i$ and consider $\Pb(\Oc \oplus \Ls_1 \oplus E')$. As $\Ls_1(1)$ has a section with derived vanishing locus $\Pb(\Oc \oplus E')$ and $\Oc(1)$ has a section with derived vanishing locus $\Pb(\Ls_1 \oplus E')$ we conclude that
\begin{align*}
&[\Pb(\Oc \oplus E') \hook \Pb(\Oc^\infty \oplus E)] \\
=& [\Pb(\Ls_1 \oplus E') \hook \Pb(\Oc^\infty \oplus E)] \\
+& e(\Ls_1) \bullet [\Pb(\Oc \oplus \Ls_1 \oplus E') \hook \Pb(\Oc^{\oplus\infty} \oplus E)] \\
+& \sum_{i,j \geq 1} a_{ij} e(\Oc(1))^{i-1} \bullet e(\Ls_1)^j \bullet [\Pb(\Ls_1 \oplus E') \hook \Pb(\Oc^{\oplus \infty} \oplus E)]
\end{align*}
or in other words
\begin{align*}
[\Pb(\Ls_1 \oplus E') \hook \Pb(\Oc^\infty \oplus E)]  &= H\big(e(\Oc(1)), e(\Ls_1)\big) \bullet \Big( [\Pb(\Oc \oplus E') \hook \Pb(\Oc^\infty \oplus E)]  \\
&- e(\Ls_1) \bullet [\Pb(\Oc \oplus \Ls_1 \oplus E') \hook \Pb(\Oc^{\oplus\infty} \oplus E)]\Big),
\end{align*}
where 
$$H(t,x) := {1 \over 1 +  \sum_{i,j \geq 1} a_{ij} t^{i-1} \bullet x^j} \in \Lb[[t,x]].$$
Hence 
\begin{align*}
&[\Pb(\Ls_1 \oplus E') \hook \Pb(\Oc^\infty \oplus E)] \\
=& (-1)^{i_1-1} \sum_{i_1 \geq 1} H\big(e(\Oc(1)), e(\Ls_1)\big)^{i_1} \bullet e(\Ls_1)^{i_1 - 1} \bullet [\Pb(\Oc^{\oplus i_1} \oplus E') \hook \Pb(\Oc^{\oplus\infty} \oplus E)],
\end{align*}
and dealing with the other $\Ls_i$ similarly, we compute that $[\Pb(E) \hook \Pb(\Oc^\infty \oplus E)]$ equals to
\begin{align*}
\sum_{i_1 \cdots i_r \geq 1}  \prod_{k=1}^r (-1)^{i_k-1} \Big( H\big(e(\Oc(1)), e(\Ls_k)\big)^{i_k} \bullet e(\Ls_k)^{i_k - 1} \Big) \bullet [\Pb(\Oc^{\oplus i_1 + \cdots + i_r}) \hook \Pb(\Oc^{\oplus\infty} \oplus E)].
\end{align*}
Therefore, by inspection 
$$j_E^{-1}\big([\Pb(E) \hook \Pb(\Oc^\infty \oplus E)]\big) = \sum_{i=0}^\infty  [\Pb^i \times X \hook \Pb^\infty \times X] \bullet H_i\big(e(\Ls_1), ..., e(\Ls_r)\big),$$
where 
$$H_i(x_1,...,x_r) \in \Lb[[x_1,...,x_r]]$$
are universal symmetric power series, $H_{r-1}$ has constant term $1$ and the other $H_i$ do not have a constant term at all. Hence we are done in the split case.

We are left with the general case. Replacing $X$ by the full flag bundle $\Fl(E)$ of $E$ we can use Lemma \ref{BivInjPullbackLem} below to reduce to the case where $E$ admits a filtration 
$$0 = E_0 \subset E_1 \subset \cdots \subset E_{r-1} \subset E_{r} = E$$
with $\Ls_i := E_{i}/E_{i-1}$ line bundles. Denoting by $E'$ the sum $\bigoplus_{i=1}^r \Ls_i$ and by $\psi_1$ and $\psi_2$ the natural inclusions $\Pb(E) \hook \Pb(E \oplus E')$ and $\Pb(E') \hook \Pb(E \oplus E')$ respectively, we may compute that
\begin{align*}
\iota_E(1_{\Pb(E)/X}) &= \iota_{E \oplus E'} \big(\psi_{1*}(1_{\Pb(E)/X})\big) \\
&= \iota_{E \oplus E'} \big(e(E'(1))\big) \\
&= \iota_{E \oplus E'} \big(e(\Ls_1(1)) \bullet \cdots \bullet e(\Ls_r(1))\big) \\
&= \iota_{E \oplus E'} \big(e(E(1))\big) \\
&= \iota_{E \oplus E'} \big(\psi_{2*}(1_{\Pb(E')/X})\big) \\
&= \iota_{E'}(1_{\Pb(E')/X})
\end{align*}
so the general case follows from the split case.
\end{proof}

We still need to prove the following simple lemma used in the above proof.

\begin{lem}\label{BivInjPullbackLem}
Let $g: Y' \to Y$ be a quasi-smooth projective morphism and suppose there exists $\eta \in \Bb^*(Y')$ such that $g_!(\eta) = 1_Y$. Then, for any morphism $X \to Y$ in $\Cc_a$, the bivariant pullback
$$g^*: \Bb^\bullet(X \to Y) \to \Bb^\bullet(X' \to Y')$$
is injective.
\end{lem}
\begin{proof}
Let $g': X' \to X$ be the pullback of $g$, and let $\alpha \in \Bb^\bullet(X \to Y)$. Then we can compute that
\begin{align*}
g'_*\big(g^*(\alpha) \bullet \eta \bullet 1_{Y' / Y} \big) &= \alpha \bullet g_*(\eta \bullet 1_{Y' / Y} ) & (\text{bivariant axiom $A_{123}$}) \\
&= \alpha \bullet g_!(\eta) \\
&= \alpha
\end{align*}
proving the claim.
\end{proof}

Next step is to make the following observation.

\begin{lem}\label{CoefficientLem2}
Suppose $X \in \Cc_a$ and let $E$ be a vector bundle of rank $r$ on $X$. Then for each $i \leq r-1$ and $j \leq r-1$ there exist unique $\alpha_{j,i}(E) \in \Bb^*(X)$ so that
\begin{align*}
u_k \Bigg(\sum_{j=0}^{r-1} e(\Oc(1))^j \bullet 1_{\Pb(E)/X} \bullet \alpha_{j,i}(E) \Bigg) &= 
\begin{cases}
1_X & \text{if k=i;} \\
0 & \text{if $k<r$ and $k \not = i$.}
\end{cases}
\end{align*} 
Moreover, $[\alpha_{j,i}(E)]$ is an invertible matrix with coefficients in $\Bb^*(X)$.
\end{lem}
\begin{proof}
Consider the matrix $A(E)$, whose entries are
\begin{align*}
A(E)_{j,i} :=& u_j\big(e(\Oc(1))^i \bullet 1_{\Pb(E)/X} \big)  \\
=& u_{i+j}\big(1_{\Pb(E)/X} \big).
\end{align*}
Then, by construction, given
$$\overline{\alpha} = 
\begin{bmatrix}
\alpha_0 \\
\vdots \\
\alpha_{r-1}
\end{bmatrix} 
\in \Bb^*(X)^{\oplus r}
$$
we see that $A(E) \overline{\alpha}$ is the vector of first $r$ coefficients of $\sum_{i=0}^{r-1} e(\Oc(1))^i \bullet 1_{\Pb(E)/X} \bullet \alpha_i$. By Lemma \ref{CoefficientLem1} $A(E)$ has units on the anti-diagonal and nilpotent elements everywhere else, so $A(E)$ is invertible by basic linear algebra. Moreover, $[\alpha_{j,i}(E)] := A(E)^{-1}$ is clearly the unique matrix with the desired property.
\end{proof}

Finally, we can prove the result we are after.

\begin{lem}\label{CoefficientLem3}
Let $X \to Y$ be a morphism in $\Cc_a$, let $u_0,...,u_{r-1} \in \Bb^\bullet(X \to Y)$, and let
$$\alpha := \sum_{i=0}^{r-1} \sum_{j=0}^{r-1} e(\Oc(1))^j \bullet 1_{\Pb(E)/X} \bullet \alpha_{j,i}(E) \bullet u_i.$$
Then $u_i(\alpha) = u_i$ for all $i < r$.
\end{lem}
\begin{proof}
This is immediate from \ref{CoefficientLem2}.
\end{proof}

We are then finally ready to prove the projective bundle formula.

\begin{proof}[Proof of Theorem \ref{GeneralPBFThm}]
By Lemma \ref{CoefficientLem3}, for each $u_0,...,u_{r-1}$, the element 
$$\alpha := \sum_{j=0}^{r-1} e(\Oc(1))^j \bullet 1_{\Pb(E)/X} \bullet \Bigg( \sum_{i=0}^{r-1} \alpha_{j,i}(E) \bullet u_i \Bigg)$$
has $u_i(\alpha) = u_i$ for $i < r$, and since $[\alpha(E)_{j,i}]$ is an invertible matrix by Lemma \ref{CoefficientLem2}, it follows that the map
$$\Pc roj: (\alpha_0,...,\alpha_{r-1}) \mapsto \sum_{j=0}^{r-1} e(\Oc(1))^j \bullet 1_{\Pb(E)/X} \bullet \alpha_j$$
is injective. On the other hand, by Lemma \ref{RCoeffLem}, there exists at most one element of $\Bb^\bullet(\Pb(E) \to Y)$ with first $r$ coefficients $u_0,...,u_{r-1}$, so the surjectivity of $\Pc roj$ follows from Lemma \ref{CoefficientLem3} as well. 
\end{proof}

\subsection{Chern classes and applications}\label{GenChernClassSubSect}

The purpose of this section is to define Chern classes of vector bundles, and to apply them to generalizing Conner--Floyd theorem and Grothendieck--Riemann--Roch theorem to all divisorial Noetherian derived schemes. We start with the fundamental definition.

\begin{defn}\label{GeneralChernClassDefn}
Let $X \in \Cc_a$ and let $E$ be a vector bundle of rank $r$ on $X$. Then by the projective bundle formula (Theorem \ref{GeneralPBFThm}) the equation
$$\sum_{i=0}^r (-1)^i e(\Oc(1))^{i} \bullet c_{r-i}(E) = 0 \in \PCob'^r(\Pb(E^\vee))$$
holds for unique $c_i(E) \in \PCob'^i(E)$ with $c_0(E)=1_X$. The element $c_i(E)$ is called the \emph{$i^{th}$ Chern class} of $E$. 
\end{defn} 

These Chern classes satisfy the expected properties, as the following result shows.

\begin{thm}\label{GeneralChernClassThm}
Define the \emph{total Chern class} of a rank $r$ vector bundle $E$ as
$$c(E) = 1 + c_1(E) + \cdots + c_r(E).$$
Then Chern classes satisfy the following properties:
\begin{enumerate}
\item \emph{naturality:} $f^* c_i(E) = c_i(f^*E)$ for $f: Y \to X$ in $\Cc_a$;
\item \emph{Whitney sum formula:} given a short exact sequence 
$$0 \to E' \to E \to E'' \to 0$$
of vector bundles on $X$, we have
$$c(E) = c(E') \bullet c(E'') \in \PCob'^*(X);$$ 
\item \emph{normalization:} if $E$ is a rank $r$ vector bundle on $X \in \Cc_a$, then $c_r(E) = e(E)$.
\end{enumerate}
\end{thm}
\begin{proof}
Naturality is trivial. To prove the Whitney sum formula, we can pull back to flag bundles to reduce to the case where both $E'$ and $E''$ admit filtrations
$$0 = E'_0 \subset E'_1 \subset \cdots \subset E'_{r'} = E'$$
and
$$0 = E''_0 \subset E''_1 \subset \cdots \subset E''_{r''} = E''$$
with graded pieces line bundles $\Ls'_i := E'_i/E'_{i-1}$ and $\Ls''_i := E''_i / E''_{i-1}$ respectively. One then argues as in Lemma \ref{ChernClassFormulaLem} to show that
$$c_i(E') = s_i\big(e(\Ls'_1), ..., e(\Ls'_{r'}) \big),$$
$$c_i(E'') = s_i\big(e(\Ls''_1), ..., e(\Ls''_{r''}) \big)$$
and 
$$c_i(E) = s_i\big(e(\Ls'_1), ..., e(\Ls'_{r'}), e(\Ls''_1), ..., e(\Ls''_{r''}) \big)$$
proving the Whitney sum formula. Normalization is trivial for line bundles, and follows for general vector bundles from the Whitney sum formula.
\end{proof}

Our next goal is to prove the Conner--Floyd theorem.
\begin{thm}\label{GeneralCFThm}
Let $\Zb_m$ be the integers considered as an $\Lb$-algebra via the multiplicative formal group law $x + y - xy$, and suppose $X \in \Cc_a$. Then the morphisms
$$\Zb_m \otimes_\Lb \PCob^*(X) \to K^0(X)$$
given by
$$[V \stackrel f \to X] \mapsto [f_* \Oc_V]$$
are isomorphisms of rings that commute with pullbacks and Gysin pushforwards. The theorem remains true if we replace $\PCob'^*(X)$ with $\Omega'^*(X)$.
\end{thm}
\begin{proof}
Now that we are armed with a good theory of Chern classes, this result can be proved in almost the same way as in \cite{annala-chern} Section 5.1 and \cite{annala-cob} Section 4.1 with one difference: we have to show that the \emph{Chern character morphisms}
$$ch: K^0(X) \to \Zb_m \otimes_\Lb \PCob^*(X)$$
defined by
$$[E] \mapsto \rank(E) - c_1(E^\vee)$$
commute with pushforwards along natural projections
$$\pi: \Pb\big(\Ls_1 \oplus \cdots \oplus \Ls_r \big) \to X$$
(this is part of \cite{annala-cob} Lemma 4.2, but the proof of the other part generalizes immediately). Applying the $K$-theoretic projective bundle formula, and the projection formula, this reduces to showing that for all $i < r$
\begin{align*}
1_X &=  ch \big(1_X \big) \\
&= ch \Big( \pi_!\big(c_1(\Oc(1))^i \big) \Big) \\
&= \pi_! \Big(ch \big(c_1(\Oc(1))^i \big) \Big)  \\
&= \pi_! \big(c_1(\Oc(1))^i \big),
\end{align*}
where the second equality follows from a simple $K$-theory computation, and the last equality follows from the fact that $ch$ preserves Chern classes and multiplications (see \cite{annala-cob} Lemma 4.1 and the discussion preceding it).  But this is easy since
\begin{align*}
c_1(\Oc(1))^i &= \prod_{k=1}^i \big(c_1(\Ls_k(1)) + c_1(\Ls_k^\vee) - c_1(\Ls_k(1)) \bullet  c_1(\Ls_k^\vee)\big) \\
&= \prod_{k=1}^i \big((1 - c_1(\Ls_k^\vee)) \bullet c_1(\Ls_k(1)) + c_1(\Ls_k^\vee)\big)
\end{align*}
which pushes forward to $1_X$ by Lemma \ref{ProjBundlesPushToUnitLem} below. The same proof goes through for $\Zb_m \otimes_\Lb \Omega^*(X)$ because all the relevant maps stay well defined.
\end{proof}

\begin{lem}\label{ProjBundlesPushToUnitLem}
Let $E$ be a vector bundle of rank $r$ on $X \in \Cc_a$ and let $\pi: \Pb(E) \to X$. Then
$$\pi_!(1_{\Pb(E)}) = 1_X \in \Zb_m \otimes_\Lb \PCob'^*(X).$$
\end{lem}
\begin{proof}
Pulling back to the flag bundle, we may assume without loss of generality that $E$ has a filtration
$$0 = E_0 \subset E_1 \subset \cdots \subset E_{r} = E$$
with $\Ls_i := E_i / E_{i-1}$ line bundles. Letting $E' = \bigoplus_{i=1}^r \Ls_i$, and noticing that
\begin{align*}
[\Pb(E) \hook \Pb(E \oplus E')]  &= e(E'(1)) \bullet 1_{\Pb(E \oplus E')} \\
&= e(\Ls_1(1)) \bullet \cdots \bullet e(\Ls_r(1)) \bullet 1_{\Pb(E \oplus E')} \\
&= e(E(1)) \bullet 1_{\Pb(E \oplus E')} \\
&= [\Pb(E') \hook \Pb(E \oplus E')]
\end{align*} 
we are reduced to the case where $E$ is a direct sum of line bundles $\Ls_1,...,\Ls_r$. But it was computed in Lemma \ref{CoefficientLem1} (use multiplicative formal group law in the computations) that
$$[\Pb(\Ls_1 \oplus \cdots \oplus \Ls_r) \to X] = \sum_{i_1,...,i_r \geq 1} \prod_{k=1}^r {\big(-c_1(\Ls_k)\big)^{i_k - 1} \over \big(1 - c_1(\Ls_k)\big)^{i_k}} \bullet [\Pb^{i_1 + \cdots + i_k - 1} \times X \to X]$$
and since
$$\sum_{i=0}^\infty {(-x)^i \over (1-x)^i} = 1-x,$$
we are reduced to showing that $[\Pb^n \times X \to X] = 1_X$ for all $n \geq 0$. But this can be proven as in the proof of \cite{Levine:2007} Lemma 4.2.3.
\end{proof}

We end by the following formal corollary of Theorem \ref{GeneralCFThm}. We will denote by $\Zb_a$ ($\Qb_a$) the integers (rationals) considered as an $\Lb$-algebra via the additive formal group law $x + y$. 

\begin{thm}\label{GeneralGRRThm}
For each $X \in \Cc_a$, there exists a Chern character morphism
$$ch_a: K^0(X) \to \Qb_a \otimes \PCob'^*(X)$$
which is a homomorphism of rings and commutes with pullbacks. Moreover, given $f: X \to Y$ projective and quasi-smooth, we have that
$$
f_! \bigl( ch_a(\alpha) \bullet \Td(\Lb_{X/Y}) \bigr) =  ch_a(f_!(\alpha))
$$
for all $\alpha \in K^0(X)$, where $\Td(\Lb_{X/Y})$ denotes the Todd class of $\Lb_{X/Y}$. Moreover, the induced map
$$ch_a: \Qb \otimes K^0(X) \to \Qb \otimes \PCob_a^*(X)$$
is an isomorphism. The result remains true if we replace $\PCob'^*(X)$ with $\Omega'^*(X)$.
\end{thm}
\begin{proof}
As in Section 5.2 of \cite{annala-chern}, this is a formal consequence of the Conner--Floyd theorem (Theorem \ref{GeneralCFThm}).
\end{proof}

\bibliographystyle{alphamod}
\bibliography{references}{}

\end{document}